\documentclass{article}
\usepackage{graphicx,geometry,amsmath,amsthm,amssymb,mathdots,authblk,hyperref} 
\usepackage{mathrsfs}
\usepackage{array}

\newtheorem{theorem}{Theorem}[section]
\newtheorem{lemma}[theorem]{Lemma}

\newtheorem{remark}[theorem]{Remark}

\numberwithin{equation}{section}

\title{The Mean Field Equation on the Tate Curve}
\author{Yaojia Sun}
\date{}

\begin{document}

\maketitle
\begin{abstract}
    In this paper, we study the spectrum of the Laplacian on the Tate curve and construct the associated Green's function as a finite sum, which can be viewed as the non-Archimedean counterpart of the Green's function on the flat torus in the Archimedean case. Moreover, we establish existence and uniqueness results for the mean field equation on this space. To address the problem, we first analyze the structure of solutions on finite quotients, and then establish existence on the Tate curve by passing to the limit along these solutions. We also prove uniqueness of solutions in a certain parameter range. Notably, the well-posedness of the solution resembles that in the Archimedean case.
\end{abstract}
\section{Introduction}
The mean field equation,
\begin{align}
\label{eq:Archimedean_mean_field_equation}
    \Delta u+\rho e^u=\rho\delta_0,\ \rho>0
\end{align}
rooted in statistical mechanics, is formally analogous to the Liouville equation. However, the latter originates from the prescribed Gaussian curvature problem and is typically formulated without a point source term on the right-hand side. A more general formulation is given by the Kazdan-Warner equation \cite{MR343206}. If $\rho=4\pi$, equation \eqref{eq:Archimedean_mean_field_equation} is related to the Chern-Simons model \cite{chen2004concentration}. On special domains such as $\Omega\subseteq\mathbb{R}^n$ or compact manifolds, the existence and uniqueness of solutions have been extensively studied in \cite{chen2003topological,MR2352519,lin2013existence,lin2011uniqueness,chen2015mean,chen2017existence,MR4366002,kuo2022even}.

The existence of solutions to this class of equations on graphs have been widely investigated. Grigor'yan, Lin, and Yang \cite{grigoryan2016kazdan} provided various conditions for the existence of the Kazdan-Warner equation on finite graphs. Ge and Jiang \cite{ge2018kazdan} further proved the existence of solutions on infinite graphs.

Huang-Lin-Yau \cite{huang2020existence} established the existence of solutions to the mean field equation on finite weighted graphs via the calculus of variations. Furthermore, Lin and Yang introduced a mean-field type heat flow on finite graphs \cite{MR4305428}. In the same year, Liu studied the mean-field equation with a positive prescribed function by using the method of Brouwer degree \cite{MR4490503}. Later, Sun and Wang promoted their results to the case of sign-changing prescribed functions \cite{MR4416135}. Subsequently, Liu and Zhang further developed the heat flow method and resolved the convergence problem for the mean-field equation with a sign-changing prescribed function \cite{MR4607589}. More recently, Li and Zhang systematically investigated the mean-field equation with a linear term by means of variational methods \cite{MR4901552}. Liu further extended these results to nonlocal operators via considering the fractional Laplacian $(-\Delta)^s$, thereby unifying and generalizing the results of \cite{MR4305428,MR4416135,MR4607589} for the case $s=1$.

Lin and Wang \cite{lin2010elliptic,lin2010function} studied the mean field equation on the flat torus $T=\mathbb{C}/(\mathbb{Z}+\mathbb{Z}\tau)$ for $\tau\in\mathbb{C}\backslash\mathbb{R}$. If $\rho\neq8m\pi$ for $m\in\mathbb{Z}$, there always exists a solution as the Leray-Schauder degree is nonzero, which is independent of the torus shape $\tau$. When $\rho=8\pi$, they demonstrated that the existence and uniqueness of solutions depend on the number of critical points of the Green's function
\begin{align}
    \Delta G(z,w)=\delta_w(z)-\frac{1}{V},
\end{align}
where $V$ denotes the volume of the torus $T$. In particular, the mean field equation admits a unique solution if and only if the Green's function possesses five critical points (otherwise it has only three).

Since these Liouville-type equations have numerous applications in quantum field theory and string theory \cite{zamolodchikov1996conformal,harlow2011analytic,benjamin2025resurgence}, and the Green’s function underlies the computation of Veneziano amplitudes, it is natural to investigate a $p$-adic analogue of this problem. Background on $p$-adic pseudo-differential operators can be found in \cite{MR3586737}.

In the genus-zero case, Huang, Stoica, Yau and Zhong \cite{huang2021greens,huang2022quadraticreciprocityfamilyadelic} defined the $p$-adic version of the flat Laplacian via regularizing Vladimirov derivation as
\begin{align*}
    D^su(x)=\int_{\mathbb{Q}_p}\frac{u(z)-u(x)}{|z-x|_p^{1+s}}dz
\end{align*}
and constructed the associated Green's function, where $dz$ is the Haar measure on the $p$-adic field $\mathbb{Q}_p$. In addition, on the upper half-plane, Hassan \cite{hassan2025padichighergreensfunctions} constructed the higher Green's function.

In the genus-one case, Huang-Rohrlich-Sun-Whyman \cite{huang2025greensfunctiontatecurve} defined a $p$-adic version of the flat Laplacian on the Tate curve as
\begin{align*}
    Du(x)=|x|_p\int_{\mathbb{Q}_p^{\times}/q^{\mathbb{Z}}}\frac{u(z)-u(x)}{|z-x|_p^2}dz,
\end{align*}
where $|q|_p<1$. One readily checks that the operator $D$ is self-adjoint and negative semi-definite under the inner product induced by the Haar measure $\frac{dx}{|x|_p}$ on $\mathbb{Q}_p^{\times}/q^{\mathbb{Z}}$, one readily verifies that it preserves locally constant functions. They proved that the Green's function only depends on $d(x,y)=\frac{|x-y|_p}{\max\{|x|_p,|y|_p\}}$ and provided a series expression for the Green's function by correcting the logarithmic singularity on the diagonal.

\begin{remark}
    $\mathbb{Q}_p/(\mathbb{Z}+\mathbb{Z}\tau)$ is not a suitable object because $\mathbb{Z}+\mathbb{Z}\tau$ is not closed in $\mathbb{Q}_p$, hence the quotient topology is not Hausdorff. In the Archimedean case, we consider the exponential map: $x\mapsto e^{2\pi ix}$, yielding $\mathbb{C}/(\mathbb{Z}+\mathbb{Z}\tau)\rightarrow\mathbb{C}^{\times}/(e^{2\pi i\tau})^\mathbb{Z}$. Based on that, the Tate curve $\mathbb{Q}_p^{\times}/q^{\mathbb{Z}}$ with $|q|_p<1$ can be viewed as the non-Archimedean analogue of the flat torus.
\end{remark}

\begin{remark}
    Since multiplication by a unit in $\mathbb{Z}_p^*$ leaves $D$ invariant, we may restrict our attention to the case $q=p^m$, and choose the fundamental domain as
$\bigcup_{k=0}^{m-1}p^k\mathbb{Z}_p^*$. One can verify that the definition is independent of the choice of the fundamental domain.
\end{remark}

For more general cases, Bradley derived the Green's function via the heat kernel on a compact domain \cite{bradley2025boundaryvalueproblemspadic} and showed the existence of the heat kernel and the Green's function for the fractional Laplacian with $s>1$ on compact manifolds \cite{bradley2025diffusionoperatorspadicanalytic}.

In this paper, we study the mean field equation and the Green's function on the Tate curve, given respectively by

\begin{equation}
\label{eq:Green's_function}
    DG(x,y)=\delta_y(x)-\frac{1}{V}\ \text{in}\ \mathbb{Q}_p^{\times}/q^{\mathbb{Z}},
\end{equation}

\begin{equation}
\label{eq:mean_field_equation}
    Du+\rho e^{u}=\rho\delta_y\ \text{in}\ \mathbb{Q}_p^{\times}/q^{\mathbb{Z}}.
\end{equation}

A standard approach to constructing the Green’s function employs the spectrum of the operator $D$ and its eigenfunctions. When $m=1$, the fundamental domain is a locally compact group, allowing us to obtain its spectrum and eigenfunctions via characters. Part of this analysis was already carried out in \cite{huang2026glimpseultrametricspectrum}. We extend these results to the general $m$ case through zero extension, although the operator $D$ is a nonlocal, and  we make minor modifications for the case where the eigenvalue is zero. Using orthogonality of the characters, we express the Green's function as a finite sum, which differs from the series formula in \cite{huang2025greensfunctiontatecurve} and is more convenient for computation and interpretation.

Let $\boldsymbol{R}=(1-p^{-1})(\boldsymbol{P}-\text{diag}(\boldsymbol{P}\boldsymbol{1}_m))$, where $(\boldsymbol{P})_{ij}=p^{-|i-j|}$ and $\text{diag}(\boldsymbol{v})$ denotes the diagonal matrix with the vector $\boldsymbol{v}$ as its diagonal entries. We have orthogonal matrix $\boldsymbol{Q}$ to diagonalize $\boldsymbol{R}$, and we define
\begin{align*}
    \phi_{h,c,n}(x)=e^{\frac{2\pi ih}{p-1}\text{ord}(a_0)}e^{2\pi i\left\{\frac{c\ln\left(\frac{x}{a_0}\right)}{p^{n}}\right\}_p},
\end{align*}
where $x=a_0+a_1p+\dots\in\mathbb{Z}_p^*$, $0\leq h\leq p-2$ and $c\in\mathbb{Z}_p^*$ with $cp^{n-1}\in\mathbb{Q}_p/\mathbb{Z}_p$. Then our first main result is as follows:

\begin{theorem}[The spectrum of $D$ and the corresponding Green's function]
    The eigenvalues of $D$ for $\mathbb{Q}_p^{\times}/p^{m\mathbb{Z}}$ are $\lambda_{0,l}\in\sigma(\boldsymbol{R})$ with $\lambda_{0,0}=0$, and $\lambda_{n,l}=-p^{n-1}-p^{n-2}+\frac{p^{-l}+p^{l+1-m}}{p}$ for $n\geq 1$, with the corresponding eigenfunctions $\phi_0^l(x)=\sum_{k=0}^{m-1}q_{k,l}\boldsymbol{1}_{p^{k}\mathbb{Z}_p^*}(x)$ and $\phi_{h,c,n}^l(x)=\phi_{h,c,n}(p^{-l}x)\boldsymbol{1}_{p^{l}\mathbb{Z}_p^*}(x)$, where $l\in\{0,1,2,\dots,m-1\}$. Furthermore, the multiplicities are $m_{0,0}=1$ (provided $\sqrt{p}+\sqrt{p^{1-m}}>2$), $m_{1,l}=p-2$, and $m_{n,l}=(p-1)^2p^{n-2}$ for $n\geq2$.

    And up to a constant $C(|y|_p)$, the corresponding Green's function \eqref{eq:Green's_function} is
    \begin{align*}
        &G(x,y)=\frac{p^{M+1}}{p^{M+1}+p^{M}-(p^{-l}+p^{l+1-m})}-\sum_{n=1}^{M}\frac{(p-1)p^n}{p^{n}+p^{n-1}-(p^{-l}+p^{l+1-m})}+\sum_{k=1}^{m-1}\frac{q_{v_p(x)k}q_{v_p(y)k}}{(1-p^{-1})\lambda_{0,k}}
    \end{align*}
    when $v_p(x)=v_p(y)=l$ and $v_p(d(x,y))=M$.
\end{theorem}

Motivated by the existence results for the Green's function \cite{huang2025greensfunctiontatecurve}, we consider the mean field equation on the finite quotient $\mathbb{Q}_{p}^{\times}/p^{m\mathbb{Z}}/(1+p^d\mathbb{Z}_{p})$ first. The operator $D$ on this finite quotient simplifies to a graph Laplacian, reducing the problem to an equation on a finite weighted graph. In this setting, we prove a universal upper bound for the solution and establish its uniqueness under a suitable condition on the parameter $\rho$. We then study the convergence of these solutions to obtain the existence result on $\mathbb{Q}_{p}^{\times}/p^{m\mathbb{Z}}$. By employing arguments similar to those used for the finite quotient, uniqueness can also also be derived. As in Archimedean case, Lin and Lucia showed that when $T=(0,a)\times(0,b)$ and $\rho\leq\min\{8\pi,\lambda_1(T)V\}$, then the solution is unique, and they conjectured that the conditions for $\rho$ holds for any torus \cite{MR2352519}. Later, Gu-Gui-Hu-Li proved their conjecture \cite{MR4366002}. The optimal upper bound for $\rho$ can be viewed as the analogoue of the $\lambda_1(T)V$ threshold in Archimedean case. Then our second main result is as follows:

\begin{theorem}
    When $\rho\in(0,\rho^*)$, then there exists a unique solution of the mean field equation \eqref{eq:mean_field_equation}, where $\rho^*=\theta_{p,m}e^{-M}$ with
    \begin{align*}
        M=\sup_{u}\left\{\max_{x}u|\ u\text{ is the radially symmetric monotone solution of \eqref{eq:mean_field_equation}}\right\}
    \end{align*}
    and $\theta_{p,m}$ is the smallest nonzero eigenvalue of $-D$. Furthermore, the solution is radially symmetric and strictly monotone.
\end{theorem}

\begin{remark}
    A rough estimate shows that
    \begin{align*}
        \theta_{p,m}V>\rho^*>\begin{cases}
        \frac{1}{4}\theta_{2,m}, &p=2;\\
    (1-2p^{-1})\theta_{p,m},\ &p>2.
    \end{cases}
    \end{align*}
    In the two extremal cases, $\rho^*$ attains its upper bound when $u$ is constant, and attains its lower bound when $u$ tends to $-\infty$ except at the point farthest from the singularity within the same stratum. Since the mean field equations considered in \cite{MR2352519} and \cite{MR4366002} are of the form
    \begin{align}
    \label{eq:their_mean_field_equation}
        \Delta u+\rho\left(\frac{e^u}{\int_{T}e^udx}-\frac{1}{V}\right)=0.
    \end{align}
    By substituting the Green's function in \eqref{eq:Archimedean_mean_field_equation}, we obtain
    \begin{align}
    \label{eq:weighted_mean_field_equation}
        \Delta v+\rho\left(e^{\rho G}e^{v}-\frac{1}{V}\right)=0,
    \end{align}
    where $v=u-\rho G$. Therefore, these two types of mean-field equations cannot be directly transformed by substituting the Green's function, because the resulting equation carries an additional weight $h(x)=e^{\rho G(x)}$. Under the setting of \eqref{eq:their_mean_field_equation}, the unique solution $u$ is constant, which coincides with the upper bound of $\rho^*$. However, this situation cannot occur here owing to the strict monotonicity of $u$.
\end{remark}

\begin{remark}
Recently, Huang and Jepsen defined such an operator on the Tate curve:
    \begin{align*}
        \bar{D}u(x)=c_p\int_{\mathbb{Q}_p^{\times}/p^{m\mathbb{Z}}}H(z,x)(u(x)-u(z))d\mu^{\times}(z),
    \end{align*}
    where $c_p=\frac{p(p-1)}{p+1}$, $d\mu^{\times}(z)=\frac{dz}{|z|_p}$ and
    \begin{align*}
        H(z,x)=\frac{|x|_p|z|_p}{|z-x|_p^2}+\frac{1}{p^{m}-1}\left(\frac{|x|_p}{|z|_p}+\frac{|z|_p}{|x|_p}\right),
    \end{align*}
    which ensures that $H(\alpha z,\alpha x)=H(z,x)$ for all $\alpha,z,x\in\mathbb{Q}_p^{\times}/p^{m\mathbb{Z}}$ (Our integral kernel is homogeneous on $\mathbb{Q}_p$). They studied the spectrum of $\bar{D}$ and showed that the Green's function is exactly the height function \cite{huang2026onelooppadicstringtheory}. Our results also hold for the operator $\bar{D}$ since it has a similar symmetric structure to $D$, one only needs to replace $\theta_{p,m}$ with the smallest nonzero eigenvalue of $\bar{D}$, which is
    \begin{align*}
        \bar{\theta}_{p,m}=\begin{cases}
            p-1,&m=1;\\
            \frac{p(p-1)(2-2\cos{\frac{2\pi}{m}})}{p^2-2p\cos{\frac{2\pi}{m}}+1},&m>1.
        \end{cases}
    \end{align*}
\end{remark}

Our goal is to connect the Green's function with the mean field equation to explore the critical point properties of the Green's function in future work. Since the results in Archimedean case were obtained indirectly through the mean-field equation, our results will provide a first step toward the non-Archimedean counterpart of this problem.

\subsection{Outline of the paper}
\hspace{\parindent}In Section 2, we study the spectrum of $D$ for the case $m=1$ by using character theory and construct the Green's function and heat kernel as a finite sum, and we extend the result to the general $m$ case.

In Section 3, we study the mean field equation on finite quotients of the Tate curve. We derive a uniform positive upper bound $M$ (depending only on $p$ and $m$) for the solution and demonstrate its central symmetry, cyclic symmetry, and monotonicity. Furthermore, we prove that the solution is radially symmetric and dependents only on $|x|_p$, $|y|_p$ and $|x-y|_p$ when $\rho\leq (1-p^{-m})e^{-M}$, which is similar to the property of the Green's function. Lastly we establish uniqueness for $\rho<\rho^*$.

In Section 4, we show that the solutions on the finite quotient converge to a solution on the Tate curve while preserving their structural properties, thereby establishing existence, and we prove that under the same condition the solution is unique.

In Section 5, we promote the results of the mean field equation to the case of elliptic curves with good reduction.

\section{Spectrum of D and the Green's function}
This section focuses on the spectrum of $D$ and its eigenfunctions. First, we consider the case $m=1$, where the fundamental domain is $\mathbb{Z}_p^*$, which is a compact multiplicative group. We analyse $D$ via characters of $\mathbb{Z}_p^*$ and construct the Green's function with a finite sum. Then we generalize it to the general $m$ case through zero extension.

\subsection{Multiplicative Characters of \texorpdfstring{$\mathbb{Z}_p^*$}{Zp*}}
Since $\mathbb{Z}_p^*\cong(\mathbb{Z}/p\mathbb{Z})^{\times}\times(1+p\mathbb{Z}_p)$, where the first factor is cyclic and the second is torsion free, it suffices to determine the characters of both components.

Let $g$ be a primitive root of $(\mathbb{Z}/p\mathbb{Z})^{\times}$. Because the dual group of $(\mathbb{Z}/p\mathbb{Z})^{\times}$ is isomorphic to itself, for $w,h\in\{0,1,\dots,p-2\}$, the characters for $(\mathbb{Z}/p\mathbb{Z})^{\times}$ are given by
\begin{align*}
    \chi_h(g^w)=e^{\frac{2\pi ih}{p-1}w}.
\end{align*}

As the dual group of $\mathbb{Z}_p$ is isomorphic to $\mathbb{Q}_p/\mathbb{Z}_p$, the characters on $\mathbb{Z}_p$ are $\chi_t(x)=e^{2\pi i\{tx\}_p}$, where $x\in\mathbb{Z}_p,t\in\mathbb{Q}_p/\mathbb{Z}_p$. And we notice that there is a continuous isomorphism
\begin{align*}
    \mathbb{Z}_p&\cong1+p\mathbb{Z}_p\\
    x&\mapsto e^{px}.
\end{align*}
So the characters of $1+p\mathbb{Z}_p$ can be pulled back from the isomorphism as
\begin{align*}
    \chi_t(x)=e^{2\pi i\left\{t\frac{\ln x}{p}\right\}_p},
\end{align*}
where $x\in1+p\mathbb{Z}_p,t\in\mathbb{Q}_p/\mathbb{Z}_p$ and $\{\cdot\}_p$ denotes the $p$-adic fractional part. Letting
\begin{align*}
    t=c_{-n}p^{-n}+c_{-n+1}p^{-n+1}+\dots+c_{-1}p^{-1}=\frac{c_{-n}+c_{-n+1}p+\dots+c_{-1}p^{n-1}}{p^n}
\end{align*}
and taking
\begin{align*}
    c=c_{-n}+c_{-n+1}p+\dots+c_{-1}p^{n-1},
\end{align*}
we rewrite the character $\chi_t(x)$ as
\begin{align*}
    \chi_{c,n+1}(x)=e^{2\pi i\left\{\frac{c\ln x}{p^{n+1}}\right\}_p}.
\end{align*}

Thus $\forall x=a_0+a_1p+\dots\in \mathbb{Z}_p^*$ with $a_0\neq 0$, the characters of $\mathbb{Z}_p^*$ are trivial character $1$ and
\begin{align}
\label{eq:character_m_1}
    \chi_{h,c,n}(x)=\chi_h(a_0)\chi_{c,n}\left(\frac{x}{a_0}\right)=e^{\frac{2\pi ih}{p-1}\text{ord}(a_0)}e^{2\pi i\left\{\frac{c\ln\left(\frac{x}{a_0}\right)}{p^{n}}\right\}_p},
\end{align}
where $h\in\{0,1,\dots,p-2\},n\in\mathbb{N_{+}}$. From \eqref{eq:character_m_1} one can easily check that the conductor of the character $\chi_{h,c,n}$ is $n$.

\subsection{Eigenvalues and eigenfunctions for \texorpdfstring{$D$}{D} when \texorpdfstring{$m=1$}{m=1}}

\begin{lemma}
\label{lem:orthogonal_characters_m=1}
    The characters $\{\chi_{h,c,n}\}_{h,c,n}$ and the trivial character $1$ are orthogonal.
\end{lemma}
\begin{proof}
First, $\langle1,1\rangle=1-p^{-1}$ and $\langle\chi_{h,c,n},1\rangle=0$ as $\chi_{h,c,n}$ is nontrivial on $\mathbb{Z}_p^*$. Next, 
    \begin{align*}
        \langle\chi_{h,c,n},\chi_{h',c',n'}\rangle&=\int_{\mathbb{Z}_p^*}\chi_{h,c,n}(x)\overline{\chi_{h',c',n'}(x)}dx\\
        &=\int_{\mathbb{Z}_p^*}e^{\frac{2\pi i(h-h')}{p-1}\text{ord}(a_0)}e^{2\pi i\left\{\left(\frac{c}{p^{n}}-\frac{c'}{p^{n'}}\right)\ln\left(\frac{x}{a_0}\right)\right\}_p}dx\\
        &=\sum_{a_0=1}^{p-1}e^{\frac{2\pi i(h-h')}{p-1}\text{ord}(a_0)}\int_{1+p\mathbb{Z}_p}e^{2\pi i\left\{\left(\frac{c}{p^{n}}-\frac{c'}{p^{n'}}\right)\ln\left(\frac{x}{a_0}\right)\right\}_p}dx\\
        &=p^{-1}\sum_{a_0=1}^{p-1}e^{\frac{2\pi i(h-h')}{p-1}\text{ord}(a_0)}\int_{\mathbb{Z}_p}e^{2\pi i\left\{\left(\frac{c}{p^{n}}-\frac{c'}{p^{n'}}\right)y\right\}_p}dy\\
        &=p^{-1}\sum_{a_0=1}^{p-1}e^{\frac{2\pi i(h-h')}{p-1}\text{ord}(a_0)}\boldsymbol{1}_{\mathbb{Z}_p}\left(\frac{c}{p^{n}}-\frac{c'}{p^{n'}}\right)\\
        &=(1-p^{-1})\delta_{h,c,n}^{h',c',n'},
    \end{align*}
thus establishing orthogonality.
\end{proof}

\begin{theorem}
\label{thm:spectrum_D_m=1}
    The eigenvalues of $D$ are $\lambda_0=0$ and $\lambda_n=-p^{n-1}-p^{n-2}+\frac{2}{p}$ for $n\geq 1$, with the corresponding eigenfunctions $\phi_0=1$ and $\phi_{h,c,n}=\chi_{h,c,n}$. The multiplicities are $m_0=1,m_{1}=p-2$ and $m_{n}=(p-1)^2p^{n-2}$ for $n\geq2$.
\end{theorem}

\begin{proof}
First, because $\ker D=\mathbb{C}$, the constant function $1$ is an eigenfunction with eigenvalue $0$. Since the conductor of $\chi_{h,c,n}$ is $n$, it is trivial on $1+p^n\mathbb{Z}_p$ but nontrivial on $1+p^{n-1}\mathbb{Z}_p$. By the first orthogonality relation of the character, $\forall s\in\{1,2,\dots,n-1\}$, we have
\begin{align*}
    \sum_{x\in(1+p^{s}\mathbb{Z}_p)/(1+p^{n}\mathbb{Z}_p)}\chi_{h,c,n}(x)=0,
\end{align*}
thus it suffices to consider $z=a_0+a_1p+\dots+a_{n-1}p^{n-1}$ and $x=b_0+b_1p+\dots+b_{n-1}p^{n-1}$. Then we have
    \begin{align*}
        D\chi_{h,c,n}(x)&=\int_{\mathbb{Z}_p^*}\frac{\chi_{h,c,n}(z)-\chi_{h,c,n}(x)}{|z-x|_p^2}dz\\        &=\mu(1+p^n\mathbb{Z}_p)\int_{\mathbb{Z}_p^*/(1+p^n\mathbb{Z}_p)}\frac{\chi_{h,c,n}(z)-\chi_{h,c,n}(x)}{|z-x|_p^2}d\delta\\
        &=p^{-n}\sum_{j=0}^{n-1}\sum_{\substack{a_i=b_i\\0\leq i\leq j-1\\a_{j}\neq b_{j}\\a_{j+1}\dots a_{n-1}}}p^{2j}(\chi_{h,c,n}(z)-\chi_{h,c,n}(x))\\
        &=-p^{-n}\left((p-2)p^{n-1}+\sum_{d=1}^{n-2}(p-1)p^{n-1-d}p^{2d}+(p-1-(-1))p^{2(n-1)}\right)\chi_{h,c,n}(x)\\
        &=-p^{-n}\left((p-2)p^{n-1}+\sum_{d=1}^{n-2}(p-1)p^{n-1+d}+p^{2n-1}\right)\chi_{h,c,n}(x)\\
        &=-p^{-n}(p^{2n-1}+p^{2n-2}-2p^{n-1})\chi_{h,c,n}(x)\\
        &=-\left(p^{n-1}+p^{n-2}-\frac{2}{p}\right)\chi_{h,c,n}(x).
    \end{align*}
Hence $\chi_{h,c,n}$ is an eigenfunction of $D$ with eigenvalue $-p^{n-1}-p^{n-2}+\frac{2}{p}$ for $n\geq 1$.

Next, for any locally constant eigenfunction $f$ in $\mathbb{Z}_p^*$, we have the Fourier transform
\begin{align*}
    \hat f(\xi)=\int_{\mathbb{Z}_p^*}f(x)\overline{\chi_{\xi}(x)}dx,
\end{align*}
then
\begin{align*}
    f(x)=\sum_{\chi_{\xi}\in\text{Char}(\mathbb{Z}_p^*)}\hat f(\xi)\chi_{\xi}(x),
\end{align*}
so we have
    \begin{align*}
        Df(x)&=\sum_{\chi_{\xi}\in\text{Char}(\mathbb{Z}_p^*)}\hat f(\xi)(D\chi_{\xi}(x))\\
        &=\sum_{\chi_{\xi}\in\text{Char}(\mathbb{Z}_p^*)}\hat f(\xi)\lambda_{\xi}\chi_{\xi}(x)\\
        &=\lambda_f\sum_{\chi_{\xi}\in\text{Char}(\mathbb{Z}_p^*)}\hat f(\xi)\chi_{\xi}(x),
    \end{align*}
thus
\begin{align*}
    \sum_{\chi_{\xi}\in\text{Char}(\mathbb{Z}_p^*)}(\lambda_{\xi}-\lambda_f)\hat f(\xi)\chi_{\xi}(x)=0.
\end{align*}
By orthogonality of eigenfunctions, we have $\lambda_{\xi}=\lambda_f$ if $\lambda_{\xi}$ belongs to the expansion of $f$. That means $\text{Char}(\mathbb{Z}_p^*)$ is a complete eigenbasis of all eigenfunctions of $D$.

Lastly, the multiplicity of eigenvalues comes from counting the number of $h$ and $c$.
\end{proof}

\subsection{Construction of the Green's function when \texorpdfstring{$m=1$}{m=1}}
Using the spectrum of $D$, we construct the Green's function as
\begin{align}
    G(x,y)=\frac{1}{1-p^{-1}}\sum_{\phi_{h,c,n}\in\text{Char}(\mathbb{Z}_p^*)\backslash\{1\}}\frac{\phi_{h,c,n}(x)\overline{\phi_{h,c,n}(y)}}{\lambda_n}+C.
\end{align}
Letting $H_n$ denote the set of characters with conductor at most $n$ and defining $C_n:=H_n\backslash H_{n-1}$. By the second orthogonality relation of the characters, we have
\begin{align}
\label{eq:second_orthogonality_relation}
    \sum_{\phi\in H_{n}}\phi(x)\overline{\phi(y)}=\boldsymbol{1}_{1+p^n\mathbb{Z}_p}(xy^{-1})|H_n|,
\end{align}
where $|H_n|=(p-1)p^{n-1}=\varphi(p^n)$. If $v_p(x-y)=M$ and $n\geq M+1$, the sum \eqref{eq:second_orthogonality_relation} vanishes.\footnote{The idea of the finite sum formula for the case $m=1$ originated from Ekam's independent study.} Therefore,
\begin{equation}
    \begin{aligned}
        G(x,y)&=\frac{1}{1-p^{-1}}\sum_{n=1}^{M+1}\sum_{\phi\in C_n}\frac{\phi(x)\overline{\phi(y)}}{\lambda_n}+C\\
        &=\frac{1}{1-p^{-1}}\left(\sum_{n=1}^{M}\frac{\varphi(p^n)-\varphi(p^{n-1})}{\lambda_n}-\frac{\varphi(p^M)}{\lambda_{M+1}}\right)+C\\
        &=-\sum_{n=1}^{M}\frac{(p-1)p^n}{p^{n}+p^{n-1}-2}+\frac{p^{M+1}}{p^{M+1}+p^{M}-2}+C,
    \end{aligned}
\end{equation}
which is expressed as a finite sum.

\subsection{General \texorpdfstring{$m$}{m} case}
For general $m$, the fundamental domain is $\bigcup_{k=0}^{m-1}p^k\mathbb{Z}_p^*$. First, we consider $\phi_{h,c,n}(p^{-l}x)\boldsymbol{1}_{p^{l}\mathbb{Z}_p^*}(x)$ for $n\geq1$. We have
\begin{equation}
    \begin{aligned}
    \label{eq:compute_eigfunction_general_m}
        &D(\phi_{h,c,n}(p^{-l}x)\boldsymbol{1}_{p^{l}\mathbb{Z}_p^*}(x))(x)\\
        =&|x|_p\int_{\bigcup_{k=0}^{m-1}p^k\mathbb{Z}_p^*}\frac{\phi_{h,c,n}(p^{-l}z)\boldsymbol{1}_{p^{l}\mathbb{Z}_p^*}(z)-\phi_{h,c,n}(p^{-l}x)\boldsymbol{1}_{p^{l}\mathbb{Z}_p^*}(x)}{|z-x|_p^2}dz\\
        =&|x|_p\int_{p^l\mathbb{Z}_p^*}\frac{\phi_{h,c,n}(p^{-l}z)-\phi_{h,c,n}(p^{-l}x)\boldsymbol{1}_{p^{l}\mathbb{Z}_p^*}(x)}{|z-x|_p^2}dz
        -\left(|x|_p\int_{\bigcup_{k\neq l}p^k\mathbb{Z}_p^*}\frac{1}{|z-x|_p^2}dz\right)\phi_{h,c,n}(p^{-l}x)\boldsymbol{1}_{p^{l}\mathbb{Z}_p^*}(x),
    \end{aligned}
\end{equation}
where the first term is
\begin{equation}
    \begin{aligned}
    \label{eq:compute_eigfunction_general_m_first_term}
        &|x|_p\int_{p^l\mathbb{Z}_p^*}\frac{\phi_{h,c,n}(p^{-l}z)-\phi_{h,c,n}(p^{-l}x)\boldsymbol{1}_{p^{l}\mathbb{Z}_p^*}(x)}{|z-x|_p^2}dz\\
        =&|p^{-l}x|_p\int_{p^l\mathbb{Z}_p^*}\frac{\phi_{h,c,n}(p^{-l}z)-\phi_{h,c,n}(p^{-l}x)\boldsymbol{1}_{p^{l}\mathbb{Z}_p^*}(x)}{|p^{-l}z-p^{-l}x|_p^2}dp^{-l}z\\
        =&|p^{-l}x|_p\int_{\mathbb{Z}_p^*}\frac{\phi_{h,c,n}(z)-\phi_{h,c,n}(p^{-l}x)\boldsymbol{1}_{p^{l}\mathbb{Z}_p^*}(x)}{|z-p^{-l}x|_p^2}dz.
    \end{aligned}
\end{equation}
When $v_p(x)=l$, equation \eqref{eq:compute_eigfunction_general_m_first_term} reduces to
    \begin{align*}
        \int_{\mathbb{Z}_p^*}\frac{\phi_{h,c,n}(z)-\phi_{h,c,n}(p^{-l}x)}{|z-p^{-l}x|_p^2}dz=(D_0\phi_{h,c,n})(p^{-l}x)=\lambda_n\phi_{h,c,n}(p^{-l}x),
    \end{align*}
while the second term is
\begin{align*}
    -\sum_{\substack{k=0\\k\neq l}}^{m-1}\frac{p^{-l}}{\max\{p^{-k},p^{-l}\}^2}p^{-k}(1-p^{-1})\phi_{h,c,n}(p^{-l}x).
\end{align*}
Then the coefficient of $\phi_{h,c,n}(p^{-l}x)$ simplifies to
    \begin{align*}
        &-\sum_{\substack{k=0\\k\neq l}}^{m-1}\frac{p^{-l}p^{-k}}{\max\{p^{-k},p^{-l}\}^2}(1-p^{-1})\\
        =&-(1-p^{-1})\left(\sum_{k=0}^{l-1}p^{-l}p^{-k}p^{2k}+\sum_{k=l+1}^{m-1}p^{-l}p^{-k}p^{2l}\right)\\
        =&-(1-p^{-1})\left(p^{-l}\sum_{k=0}^{l-1}p^{k}+\sum_{k=l+1}^{m-1}p^{-(k-l)}\right)\\
        =&\frac{p^{-l}+p^{l+1-m}-2}{p},
    \end{align*}
thus \eqref{eq:compute_eigfunction_general_m} becomes
    \begin{align*}
        &-(p^{n-1}+p^{n-2}-\frac{2}{p}-\frac{p^{-l}+p^{l+1-m}-2}{p})\phi_{h,c,n}(p^{-l}x)\\
        =&-(p^{n-1}+p^{n-2}-\frac{p^{-l}+p^{l+1-m}}{p})\phi_{h,c,n}(p^{-l}x).
    \end{align*}
When $v_p(x)\neq l$, \eqref{eq:compute_eigfunction_general_m_first_term} becomes
\begin{align*}
    \frac{|p^{-l}x|_p}{\max\{1,|p^{-l}x|_p\}^2}\int_{\mathbb{Z}_p^*}\phi_{h,c,n}(z)dz=0
\end{align*}
since $\phi_{h,c,n}$ is orthogonal to $\phi_0=1$ by Lemma \ref{lem:orthogonal_characters_m=1}, and the second term vanishes. Denoting $\lambda_{n,l}=-p^{n-1}-p^{n-2}+\frac{p^{-l}+p^{l+1-m}}{p}$, we obtain
\begin{align*}
    D(\phi_{h,c,n}(p^{-l}x)\boldsymbol{1}_{p^{l}\mathbb{Z}_p^*}(x))(x)=\lambda_{n,l}\phi_{h,c,n}(p^{-l}x)\boldsymbol{1}_{p^{l}\mathbb{Z}_p^*}(x),
\end{align*}
which means $\phi_{h,c,n}(p^{-l}x)\boldsymbol{1}_{p^{l}\mathbb{Z}_p^*}(x)$ is the eigenfunction of $D$ with the eigenvalue $\lambda_{n,l}$.

When $n=0$, we consider $\boldsymbol{1}_{p^{l}\mathbb{Z}_p^*}(x)$, similarly, we have
\begin{align*}
    D\boldsymbol{1}_{p^{l}\mathbb{Z}_p^*}(x)=|p^{-l}x|_p\int_{\mathbb{Z}_p^*}\frac{1-\boldsymbol{1}_{p^{l}\mathbb{Z}_p^*}(x)}{|z-p^{-l}x|_p^2}dz+\frac{p^{-l}+p^{l+1-m}-2}{p}\boldsymbol{1}_{p^{l}\mathbb{Z}_p^*}(x),
\end{align*}
so when $v_p(x)=l$, the first part vanishes, and when $v_p(x)\neq l$, the first part becomes
\begin{align*}
    \frac{|p^{-l}x|_p(1-p^{-1})}{\max\{1,|p^{-l}x|_p\}^2}=(1-p^{-1})p^{-|v_p(x)-l|},
\end{align*}
denoting this vector by $\boldsymbol{r}_l=\begin{pmatrix}  r_{0l}\\ r_{1l}\\ \vdots \\ r_{m-1,l}\end{pmatrix}$. We want to use the linear combination of them to construct the eigenvalue and eigenfunction:
\begin{align*}
    D\sum_{l=0}^{m-1}c_l\boldsymbol{1}_{p^{l}\mathbb{Z}_p^*}(x)=\sum_{l=0}^{m-1}c_l\boldsymbol{r}_l=(\boldsymbol{r}_0,\boldsymbol{r}_1,\dots,\boldsymbol{r}_{m-1})\begin{pmatrix}  c_{0}\\ c_{1}\\ \vdots \\ c_{m-1}\end{pmatrix}=\boldsymbol{R}\boldsymbol{c},
\end{align*}
where we require that
\begin{align*}
    D\sum_{l=0}^{m-1}c_l\boldsymbol{1}_{p^{l}\mathbb{Z}_p^*}(x)=\lambda\sum_{l=0}^{m-1}c_l\boldsymbol{1}_{p^{l}\mathbb{Z}_p^*}(x)=\lambda\boldsymbol{c},
\end{align*}
so we only need to solve
\begin{align*}
    \boldsymbol{R}\boldsymbol{c}=\lambda\boldsymbol{c},
\end{align*}
where $\boldsymbol{R}=(1-p^{-1})(\boldsymbol{P}-\text{diag}(\boldsymbol{P}\boldsymbol{1}_m))$ with rank$(\boldsymbol{R})=m-1$ and $(\boldsymbol{P})_{ij}=p^{-|i-j|}$. Let the eigenvalues of $\boldsymbol{R}$ be $\lambda_{0,0}=0,\lambda_{0,1},\dots,\lambda_{0,m-1}$ and $\boldsymbol{Q}^{-1}\boldsymbol{R}\boldsymbol{Q}=\boldsymbol{\Lambda}_0$ for an orthogonal matrix $\boldsymbol{Q}$. Then $\boldsymbol{R}\boldsymbol{Q}=\boldsymbol{\Lambda}_0\boldsymbol{Q}=(\lambda_{0,0}\boldsymbol{q}_0,\dots,\lambda_{0,m-1}\boldsymbol{q}_{m-1})$. So only the constant function $1$ is an eigenfunction of $D$ with eigenvalue $0$.
Denoting
\begin{align*}
    \phi_0^l(x)=\boldsymbol{q}_{l}(x)=\sum_{k=0}^{m-1}q_{k,l}\boldsymbol{1}_{p^{k}\mathbb{Z}_p^*}(x)
\end{align*}
and
\begin{align*}
    \phi_{h,c,n}^l(x)=\phi_{h,c,n}(p^{-l}x)\boldsymbol{1}_{p^{l}\mathbb{Z}_p^*}(x),
\end{align*}
we have the following lemma:

\begin{lemma}
\label{lem:orthogonal_characters_general_m}
    The sets $\{\phi_0^l\}_l,\{\phi_{h,c,n}^l\}_{h,c,n,l}$ form a complete orthogonal eigenbasis for $D$.
\end{lemma}
\begin{proof}
The orthogonality of $\{\phi_0^l\}_l$ follows from the symmetry of $\boldsymbol{R}$. For $\phi_{h,c,n}^l$, by Lemma \ref{lem:orthogonal_characters_m=1} we have

    \begin{align*}        \langle\phi_{h,c,n}^{l},\phi_{h',c',n'}^{l'}\rangle&=\delta_{l}^{l'}\langle\phi_{h,c,n}^{l},\phi_{h',c',n'}^{l}\rangle\\
        &=\delta_{l}^{l'}\int_{p^l\mathbb{Z}_p^*}\phi_{h,c,n}(p^{-l}x)\overline{\phi_{h,c,n}(p^{-l}x)}\frac{dx}{|x|_p}\\
        &=\delta_{l}^{l'}\int_{\mathbb{Z}_p^*}\phi_{h,c,n}(x)\overline{\phi_{h,c,n}(x)}dx\\
        &=\delta_{l}^{l'} \langle\phi_{h,c,n},\phi_{h',c',n'}\rangle\\
        &=(1-p^{-1})\delta_{h,c,n,l}^{h',c',n',l'},
    \end{align*}
which is also orthogonal. Finally, $\langle\phi_{h,c,n}^{l},\phi_0^{l'}\rangle=0$ is trivial because $\phi_{h,c,n}^{l}$ is not a constant function on any stratum $p^k\mathbb{Z}_p^*$.

For any eigenfunction $f$, we have the decomposition $f=f_c+f_0+f_1+\dots+f_{m-1}$, where $f_c$ is constant on each stratum and $\text{supp}f_l\subseteq p^l\mathbb{Z}_p^*$ with zero DC component (i.e., the Fourier expansion of $f_l$ does not contain the trivial character). We can define the Fourier transform of $f_l$ by pulling it back to $\mathbb{Z}_p^*$ as $\widehat{f_l}(\xi):=\widehat{f_l(p^{-l}x)}(\xi)$. Then we can expand $f$ by Fourier series, and the characters are exactly the eigenfunctions. Similar to the discussion of $m=1$, this establishes completeness.
\end{proof}

Lastly, we handle the multiplicity of the eigenvalues of the matrix $\boldsymbol{R}$. Which is equivalent to consider $\boldsymbol{L}=\text{diag}(\boldsymbol{P}\boldsymbol{1}_m)-\boldsymbol{P}$.
\begin{lemma}
\label{lem:distinct_eigenvalues_symmetric_tridiagonal_matrix}
    The eigenvalues of an irreducible symmetric tridiagonal matrix are distinct.
\end{lemma}
\begin{proof}
    Suppose
    \begin{align*}
        \boldsymbol{T}=\begin{pmatrix}
            a_1&b_1&\\
            b_1&a_2&b_2\\
            &b_2&\ddots&\ddots\\
            & &\ddots&\ddots&b_{n-1}\\
            &&&b_{n-1}&a_n
        \end{pmatrix}
    \end{align*}
    with $b_i\neq0$ be an irreducible symmetric tridiagonal matrix. Consider $(\boldsymbol{T}-\lambda\boldsymbol{I})\boldsymbol{x}=\boldsymbol{0}$, which is
    \begin{align*}
        \begin{cases}
            (a_1-\lambda)x_1+b_1x_2=0,\\
            b_{i-1}x_{i-1}+(a_i-\lambda)x_i+b_ix_{i+1}=0,&i=2,\dots,n-1;\\
            b_{n-1}x_{n-1}+(a_n-\lambda)x_n=0.
        \end{cases}
    \end{align*}
    Since $b_i\neq0$, fixing $x_1$, uniquely determines the remaining variables. This implies that $\dim\ker(\boldsymbol{T}-\lambda\boldsymbol{I})\leq1$. Then the eigenvalues of $\boldsymbol{T}$ are distinct.
\end{proof}

\begin{lemma}
\label{lem:distinct_eigenvalues_tridiagonal_matrix}
    Let 
    \begin{align*}
        \boldsymbol{T}=\begin{pmatrix}
            a_1&b_1&\\
            c_1&a_2&b_2\\
            &c_2&\ddots&\ddots\\
            & &\ddots&\ddots&b_{n-1}\\
            &&&c_{n-1}&a_n
        \end{pmatrix}
    \end{align*}
    be a tridiagonal matrix with $b_ic_i>0$, then the eigenvalues of $\boldsymbol{T}$ are distinct.
\end{lemma}
\begin{proof}
    By Lemma \ref{lem:distinct_eigenvalues_symmetric_tridiagonal_matrix}, it suffices to find a invertible diagonal matrix $\boldsymbol{D}=\text{diag}(d_1,\dots,d_n)$ to symmetrize $\boldsymbol{T}$. So we need
    \begin{align*}
        (\boldsymbol{D}^{-1}\boldsymbol{T}\boldsymbol{D})_{i,i+1}=(\boldsymbol{D}^{-1}\boldsymbol{T}\boldsymbol{D})_{i+1,i}
    \end{align*}
    which means
    \begin{align*}
        d_i^{-1}b_id_{i+1}=d_{i+1}^{-1}c_id_{i}
    \end{align*}
    therefore, we have
    \begin{align*}
        \frac{d_{i+1}}{d_i}=\sqrt{\frac{c_i}{b_i}}
    \end{align*}
    So if we fix $d_1\neq 0$, we can construct such $\boldsymbol{D}$.
\end{proof}

\begin{theorem}
    When $\sqrt{p}+\sqrt{p^{1-m}}>2$, then the eigenvalues of $\boldsymbol{L}$ are distinct.
\end{theorem}
\begin{proof}
Consider $\boldsymbol{L}\boldsymbol{x}=\lambda\boldsymbol{x}$, which is equivalent to $\boldsymbol{P}\boldsymbol{x}=(\text{diag}(\boldsymbol{P}\boldsymbol{1}_m)-\lambda\boldsymbol{I})\boldsymbol{x}$, since
\begin{align*}
    \boldsymbol{P}^{-1}=\frac{1}{1-p^{-2}}\begin{pmatrix}
            1&-p^{-1}&\\
            -p^{-1}&1+p^{-2}&-p^{-1}\\
            &-p^{-1}&\ddots&\ddots\\
            & &\ddots&1+p^{-2}&-p^{-1}\\
            &&&-p^{-1}&1
        \end{pmatrix}
\end{align*}
then $\boldsymbol{P}^{-1}(\text{diag}(\boldsymbol{P}\boldsymbol{1}_m)-\lambda\boldsymbol{I})\boldsymbol{x}=\boldsymbol{x}$.

And by Gershgorin circle theorem, we have
\begin{align*}
    0\leq\sigma(\boldsymbol{L})\leq 2\max_i\sum_{j\neq i}p^{-|i-j|}=2\max_i\frac{2-p^{-i}+p^{i+1-m}}{p-1}\leq 2\frac{2-2p^{\frac{m-1}{2}}}{p-1}
\end{align*}
Note that
\begin{align*}
    \min_i\sum_{j}p^{-|i-j|}=\min_i\frac{p+1-p^{-i}-p^{i+1-m}}{p-1}=\frac{p-p^{1-m}}{p-1}
\end{align*}
so when $2\frac{2-2p^{\frac{m-1}{2}}}{p-1}<\frac{p-p^{1-m}}{p-1}$, i.e.
\begin{align*}
    p>\left(2-p^{\frac{m-1}{2}}\right)^2\Leftrightarrow \sqrt{p}+\sqrt{p^{1-m}}>2
\end{align*}
then the diagonal elements in $\text{diag}(\boldsymbol{P}\boldsymbol{1}_m)-\lambda\boldsymbol{I}$ are greater than $0$. Then $\boldsymbol{P}^{-1}(\text{diag}(\boldsymbol{P}\boldsymbol{1}_m)-\lambda\boldsymbol{I})$ satisfies the condition in Lemma \ref{lem:distinct_eigenvalues_tridiagonal_matrix}. Then
\begin{align*}
    \dim\ker(\boldsymbol{L}-\lambda\boldsymbol{I})=\dim\ker(\text{diag}(\boldsymbol{P}\boldsymbol{1}_m)-\boldsymbol{P}-\lambda\boldsymbol{I})=\dim\ker(\boldsymbol{P}^{-1}(\text{diag}(\boldsymbol{P}\boldsymbol{1}_m)-\lambda\boldsymbol{I})-\boldsymbol{I})\leq1
\end{align*}
Thus the eigenvalues of $\boldsymbol{L}$ are distinct.    
\end{proof}

Thus we have proved:

\begin{theorem}
\label{thm:spectrum_D_general_m}
    The eigenvalues of $D$ are $\lambda_{0,l}\in\sigma(\boldsymbol{R})$ with $\lambda_{0,0}=0$, and $\lambda_{n,l}=-p^{n-1}-p^{n-2}+\frac{p^{-l}+p^{l+1-m}}{p}$ for $n\geq 1$, with the corresponding eigenfunctions $\phi_0^l(x)=\sum_{k=0}^{m-1}q_{k,l}\boldsymbol{1}_{p^{k}\mathbb{Z}_p^*}(x)$ and $\phi_{h,c,n}^l(x)=\phi_{h,c,n}(p^{-l}x)\boldsymbol{1}_{p^{l}\mathbb{Z}_p^*}(x)$, where $l\in\{0,1,2,\dots,m-1\}$. And the multiplicities are $m_{0,0}=1$ (suppose $\sqrt{p}+\sqrt{p^{1-m}}>2$), $m_{1,l}=p-2$, and $m_{n,l}=(p-1)^2p^{n-2}$ for $n\geq2$.
\end{theorem}

\begin{remark}
    $\lambda_{n,l}$ may not be disjoint as $\lambda_{n,l}=\lambda_{n,m-l}$ for $0\leq l\leq m-1,n\geq1$.
\end{remark}

Then the Green's function is
\begin{align}
    G(x,y)=\frac{1}{1-p^{-1}}\left(\sum_{k=0}^{m-1}\sum_{\phi_n^k\in\text{Char}(p^k\mathbb{Z}_p^*)\backslash\{1\}}\frac{\phi_n^k(x)\overline{\phi_n^k(y)}}{\lambda_{n,k}}+\sum_{k=1}^{m-1}\frac{\phi_0^k(x)\overline{\phi_0^k(y)}}{\lambda_{0,k}}\right)+C(|y|_p).
\end{align}
Similarly, letting $H_n^{l}$ denote the set of characters generated from $p^l\mathbb{Z}_p^*$ with conductor at most $n$ and defining $C_n^l:=H_n^l\backslash H_{n-1}^l$. Considering
\begin{align*}
    \sum_{k=0}^{m-1}\sum_{\phi^k\in H_{n}^k}\phi^k(x)\overline{\phi^k(y)}.
\end{align*}
When $v_p(x)\neq v_p(y)$ (i.e. $v_p(d(x,y))=0$), the sum vanishes. 
When $v_p(x)=v_p(y)=l$ (i.e. $v_p(d(x,y))\geq0$), then
    \begin{align*}
            \sum_{\phi^{l}\in H_{n}^l}\phi^{l}(x)\overline{\phi^{l}(y)}&=\sum_{\phi^0\in H_{n}^0}\phi^0(p^{-l}x)\overline{\phi^0(p^{-l}y)}\\
            &=\boldsymbol{1}_{1+p^n\mathbb{Z}_p}(xy^{-1})\varphi(p^n).
    \end{align*}
So when $v_p(d(x,y))=M$ and $v_p(x)=v_p(y)=l$, we have the Green's function
\begin{equation}
    \begin{aligned}
        G(x,y)=&\frac{1}{1-p^{-1}}\left(\sum_{n=1}^{M+1}\sum_{\phi\in C_n^l}\frac{\phi(x)\overline{\phi(y)}}{\lambda_{n,l}}+\sum_{k=1}^{m-1}\frac{\phi_0^k(x)\overline{\phi_0^k(y)}}{\lambda_{0,k}}\right)+C(|y|_p)\\
        =&\frac{1}{1-p^{-1}}\left(\sum_{n=1}^{M}\frac{\varphi(p^n)-\varphi(p^{n-1})}{\lambda_{n,l}}-\frac{\varphi(p^M)}{\lambda_{M+1,l}}+\sum_{k=1}^{m-1}\frac{q_{v_p(x)k}q_{v_p(y)k}}{\lambda_{0,k}}\right)+C(|y|_p)\\
        =&-\sum_{n=1}^{M}\frac{(p-1)p^n}{p^{n}+p^{n-1}-(p^{-l}+p^{l+1-m})}+\frac{p^{M+1}}{p^{M+1}+p^{M}-(p^{-l}+p^{l+1-m})}\\
        &+\sum_{k=1}^{m-1}\frac{q_{v_p(x)k}q_{v_p(y)k}}{(1-p^{-1})\lambda_{0,k}}+C(|y|_p).
    \end{aligned}
\end{equation}
Finally, we verify that $G(x,y)$ is actually the Green's function:

When $x\neq y$,
    \begin{align*}
        DG(x,y)&=D\frac{1}{1-p^{-1}}\left(\sum_{k=0}^{m-1}\sum_{\phi_n^k\in\text{Char}(p^k\mathbb{Z}_p^*)\backslash\{1\}}\frac{\phi_n^k(x)\overline{\phi_n^k(y)}}{\lambda_{n,k}}+\sum_{k=1}^{m-1}\frac{\phi_0^k(x)\overline{\phi_0^k(y)}}{\lambda_{0,k}}\right)\\
        &=\frac{1}{1-p^{-1}}\left(\sum_{k=0}^{m-1}\sum_{\phi_n^k\in\text{Char}(p^k\mathbb{Z}_p^*)\backslash\{1\}}\phi_n^k(x)\overline{\phi_n^k(y)}+\sum_{k=1}^{m-1}\phi_0^k(x)\overline{\phi_0^k(y)}\right)\\
        &=\frac{1}{1-p^{-1}}\left(\left(\sum_{n=1}^{M}(\varphi(p^n)-\varphi(p^{n-1}))-\varphi(p^M)\right)\delta_{v_p(x)}^{v_p(y)}+\sum_{k=1}^{m-1}q_{v_p(x)k}q_{v_p(y)k}\right)\\
        &=\frac{1}{1-p^{-1}}\left(-\varphi(1)\delta_{v_p(x)}^{v_p(y)}+\sum_{k=0}^{m-1}q_{v_p(x)k}q_{v_p(y)k}-q_{v_p(x)0}q_{v_p(y)0}\right)\\
        &=\frac{1}{1-p^{-1}}\left(-\delta_{v_p(x)}^{v_p(y)}+\delta_{v_p(x)}^{v_p(y)}-\frac{1}{m}\right)\\
        &=-\frac{1}{m(1-p^{-1})}
    \end{align*}
as $\boldsymbol{Q}$ is an orthogonal matrix and $\boldsymbol{q}_0=\frac{1}{\sqrt{m}}\boldsymbol{1}$.

Then $DG(x,y)+\frac{1}{m(1-p^{-1})}$ is a distribution supported on the diagonal. And we notice that
    \begin{align*}
        &\int_{\bigcup_{k=0}^{m-1}p^k\mathbb{Z}_p^*}DG(x,y)+\frac{1}{m(1-p^{-1})}dx\\
        =&\frac{1}{1-p^{-1}}\left(\sum_{k=0}^{m-1}\sum_{\phi_n^k\in\text{Char}(p^k\mathbb{Z}_p^*)\backslash\{1\}}\overline{\phi_n^k(y)}\int_{p^k\mathbb{Z}_p^*}\phi_n^k(x)dx+\sum_{k=1}^{m-1}\overline{\phi_0^k(y)}\int_{\mathbb{Z}_p^*}\phi_0^k(x)dx\right)+1\\
        =&1.
    \end{align*}
So $G(x,y)$ satisfies $DG(x,y)=\delta_y(x)-\frac{1}{m(1-p^{-1})}$.

Furthermore, we can compute the heat kernel by using spectrum of $D$ as
\begin{equation}
    \begin{aligned}
        H(t,x,y)=&\frac{1}{1-p^{-1}}\left(\sum_{n=1}^{M+1}\sum_{\phi\in C_n^l}e^{\lambda_{n,l}t}\phi(x)\overline{\phi(y)}+\sum_{k=1}^{m-1}e^{\lambda_{0,k}t}\phi_0^k(x)\overline{\phi_0^k(y)}\right)\\
        =&\frac{1}{1-p^{-1}}\left(\sum_{n=1}^{M}e^{\lambda_{n,l}t}(\varphi(p^n)-\varphi(p^{n-1}))-e^{\lambda_{M+1,l}t}\varphi(p^M)+\sum_{k=1}^{m-1}e^{\lambda_{0,k}t}q_{v_p(x)k}q_{v_p(y)k}\right)\\
        =&(p-1)\sum_{n=1}^{M}e^{-(p^{n-1}+p^{n-2}-\frac{p^{-l}+p^{l+1-m}}{p})t}p^{n-1}-e^{-(p^{M}+p^{M-1}-\frac{p^{-l}+p^{l+1-m}}{p})t}p^M\\
        &+\frac{p}{p-1}\sum_{k=1}^{m-1}e^{\lambda_{0,k}t}q_{v_p(x)k}q_{v_p(y)k},
    \end{aligned}
\end{equation}
which is also an important fundamental quantity.

\begin{remark}
    Because $\mathbb{Q}_p^{\times}/p^{m\mathbb{Z}}\cong p^{\mathbb{Z}/m\mathbb{Z}}\times\mathbb{Z}_p^*$ is a multiplicative group, the characters of $p^{\mathbb{Z}/m\mathbb{Z}}$ are given by
    \begin{align*}
        \chi_{k}(p^l)=e^{\frac{2\pi ik}{m}l}
    \end{align*}
    for $k,l\in\{0,1,\dots,m-1\}$. So $\forall\ x=p^l(a_0+a_1p+\dots)\in\mathbb{Q}_p^{\times}/p^{m\mathbb{Z}}$ with $a_0\neq0$, we have the character
    \begin{align*}
        \chi_{k,h,c,n}(x)=\chi_{k}(p^l)\chi_h(a_0)\chi_{c,n}\left(\frac{x}{p^la_0}\right)=e^{\frac{2\pi ik}{m}l}e^{\frac{2\pi ih}{p-1}\text{ord}(a_0)}e^{2\pi i\left\{\frac{c\ln\left(\frac{x}{p^la_0}\right)}{p^{n}}\right\}_p}.
    \end{align*}
    One can check that the eigenfunctions are the linear transformation of the characters on each stratum. So the orthogonality and completeness can be easily obtained by the orthogonality of characters and Fourier transform.
\end{remark}

\begin{remark}
    For general $s$ case, we define the operator $D^s$ on the Tate curve as
    \begin{align*}
        D^su(x)=|x|_p\int_{\mathbb{Q}_p^{\times}/q^{\mathbb{Z}}}\frac{u(z)-u(x)}{|z-x|_p^{1+s}}dz,
    \end{align*}
    which is also self-adjoint, negative semi-definite and preserves locally constant functions. One can check that all our computations carry over to $D^s$, simply by replacing $2$ with $1+s$, so that the eigenfunctions of $D^s$ are identical to those of $D$. Hence, we may use the same method to determine the spectrum of $D^s$, and then obtain its Green's function and heat kernel as a finite sum formula.
\end{remark}

\section{Mean field equation on the finite quotient}
In this section, we consider the mean field equation on the Tate curve
\begin{align}
\label{eq:mean_field_equation_Tate_curve}
    Du+\rho e^u=\rho\delta_y \text{ in }\mathbb{Q}_{p}^{\times}/p^{m\mathbb{Z}}.
\end{align}
To study this equation, we first consider this equation on a finite quotient 
\begin{align}
\label{eq:mean_field_equation_finite_quotient}
    Du+\rho e^u=\rho\delta_y \text{ in }\mathbb{Q}_{p}^{\times}/p^{m\mathbb{Z}}/(1+p^d\mathbb{Z}_{p}).
\end{align}
i.e. by truncating at the $d$-th $p$-adic digits. Points are of the form $x=p^l(b_0+b_1p+\dots+b_{d-1}p^{d-1})$ with $l\in\{0,1,\dots,m-1\}$, $b_0\in\{1,2,\dots,p-1\}$ and $b_i\in\{0,1,\dots,p-1\}$ for $1\leq i\leq d-1$. Because the Haar measure on $\mathbb{Q}_{p}^{\times}/p^{m\mathbb{Z}}/(1+p^d\mathbb{Z}_{p})$ is $\mu^{\times}(p^k(1+p^d\mathbb{Z}_{p}))d\delta=p^{-d}d\delta$, where $d\delta$ is the counting measure, denote $z=p^k(a_0+a_1p+\dots+a_{d-1}p^{d-1})$, we have
\begin{equation}
\begin{aligned}
\label{eq:compute_Du}
Du(x)&=|x|_{p}\int_{\mathbb{Q}_{p}^{\times}/p^{m\mathbb{Z}}/(1+p^d\mathbb{Z}_{p})}|z|_{p}\frac{u(z)-u(x)}{|z-x|_{p}^2}p^{-d}d\delta\\
&=p^{-d}\sum_{\substack{0\leq k\leq m-1\\1\leq a_0\leq p-1\\0\leq a_i\leq p-1\\1\leq i\leq d-1}}\frac{|z|_{p}|x|_{p}}{|z-x|_{p}^2}(u(z)-u(x)).
\end{aligned}    
\end{equation}
So the flat Laplacian on the quotient space of the Tate curve reduces to a graph Laplacian on a finite weighted graph. The existence of this solution is guaranteed based on \cite{huang2020existence}. Let $\boldsymbol{u}=u(p^l\sum_{i=0}^{d-1}b_ip^i)$, $e^{\boldsymbol{u}}=e^{u(p^l\sum_{i=0}^{d-1}b_ip^i)}$, and define $\boldsymbol{A}=p^{-d}\boldsymbol{B}$ as the coefficient matrix of $\boldsymbol{u}$. Note that $\delta_y$ becomes $p^{d}\boldsymbol{e}_y$ because the Haar measure is $p^{-d}d\delta$. Then \eqref{eq:mean_field_equation_finite_quotient} can be rewritten as a matrix equation
\begin{align}
\label{eq:matrix_equation}
\boldsymbol{A}\boldsymbol{u}+\rho e^{\boldsymbol{u}}=\rho p^{d}\boldsymbol{e}_y.
\end{align}
From \eqref{eq:compute_Du}, $\boldsymbol{A}\in\mathbb{R}^{m(p-1)p^{d-1}\times m(p-1)p^{d-1}}$ is a symmetric matrix with zero row and colum sums. Moreover, $\boldsymbol{A}$ is a block matrix where each block is a symmetric circulant matrix, i.e. $\boldsymbol{A}=(\boldsymbol{A}_{ij})$, $1\leq i,j\leq m$ with $\boldsymbol{B}_{ij}=p^{-|i-j|}\boldsymbol{1}\boldsymbol{1}^\top$ for $i\neq j$.

\begin{remark}
    Since the finite quotient can converge to the Tate curve in the sence of Gromov-Hausdorff distance as $d\rightarrow\infty$, the graph Laplacian on this finite graph can converge to the flat Laplacian on the Tate curve. Thus \eqref{eq:mean_field_equation_finite_quotient} serves as a discrete approximation of \eqref{eq:mean_field_equation_Tate_curve}.
\end{remark}

\subsection{Structure of the solution}

\begin{lemma}
\label{lem:rough_upperbound}
If $\boldsymbol{u}$ solves \eqref{eq:matrix_equation}, then $u_i\leq d\ln p, \forall i$. The equality holds if and only if $p=2,m=1$.
\end{lemma}
\begin{proof}
From \eqref{eq:matrix_equation}, we have
\begin{align}
\label{eq:matrix_equation_each_row}
\sum_{j}a_{ij}u_j+\rho e^{u_i}=\rho p^{d}\delta_{iy},\ \forall i.
\end{align}
Summing over $i$ yields
\begin{align*}
\sum_{i,j}a_{ij}u_j+\rho \sum_{i}e^{u_i}=\rho p^{d},
\end{align*}
since $\sum_{i,j}a_{ij}u_j=\sum_{j}u_j\sum_{i}a_{ij}=0$, we obtain
\begin{align*}
\sum_{i}e^{u_i}=p^{d},
\end{align*}
then $e^{u_i}\leq p^{d}$, thus $u_i\leq d\ln p$ for all $i$. And the equality holds if and only if $p=2,m=1$ as there only has one term in the sum.
\end{proof}

\begin{lemma}
\label{lem:u_y_smallest}
$u_y$ is the only smallest one in $\boldsymbol{u}$.
\end{lemma}
\begin{proof}
Let $m=\min_iu_i$. Suppose $\exists\ i\neq y$, s.t. $u_i=m$. From \eqref{eq:matrix_equation_each_row} with $i\neq y$,
\begin{align*}
    \sum_{j}a_{ij}u_j+\rho e^{u_i}=0,
\end{align*}
then $\rho e^{u_i}+a_{ii}u_i=-\sum_{j\neq i}a_{ij}u_j\leq -m\sum_{j\neq i}a_{ij}=ma_{ii}$, so we have
\begin{align*}
    \rho e^{u_i}\leq a_{ii}(u_i-m)=0,
\end{align*}
a contradiction.
\end{proof}

\begin{lemma}
\label{lem:central_symmetric}
If $\boldsymbol{u}$ solves $\boldsymbol{A}\boldsymbol{u}+\rho e^{\boldsymbol{u}}=\boldsymbol{b}$, then $\boldsymbol{J}\boldsymbol{u}$ solves $\boldsymbol{A}\boldsymbol{v}+\rho e^{\boldsymbol{v}}=\boldsymbol{J}\boldsymbol{b}$, where $\boldsymbol{J}=\begin{pmatrix}  & & & 1 \\  & & 1 & \\ & \iddots & &\\ 1 & & &\end{pmatrix}$ is the anti-diagonal permutation matrix.
\end{lemma}
\begin{proof}
Since $\boldsymbol{A}$ is a symmetric circulant matrix, it is necessarily centrosymmetric matrix, implying that $\boldsymbol{A}\boldsymbol{J}=\boldsymbol{J}\boldsymbol{A}$. Given that $\boldsymbol{A}\boldsymbol{u}+\rho e^{\boldsymbol{u}}=\boldsymbol{b}$, it follows that $\boldsymbol{J}\boldsymbol{A}\boldsymbol{u}+\rho \boldsymbol{J}e^{\boldsymbol{u}}=\boldsymbol{J}\boldsymbol{b}$. This can be rewritten as $\boldsymbol{A}\boldsymbol{J}\boldsymbol{u}+\rho e^{\boldsymbol{J}\boldsymbol{u}}=\boldsymbol{J}\boldsymbol{b}$, therefore, $\boldsymbol{J}\boldsymbol{u}$ is a solution of $\boldsymbol{A}\boldsymbol{v}+\rho e^{\boldsymbol{v}}=\boldsymbol{J}\boldsymbol{b}$.
\end{proof}

Let $\boldsymbol{u}=\begin{pmatrix}  \boldsymbol{u}_1\\ \boldsymbol{u}_2\\ \vdots \\ \boldsymbol{u}_m\end{pmatrix}$ where $\boldsymbol{u}_i\in\mathbb{R}^{(p-1)p^{d-1}}$, $\boldsymbol{\sigma}=\begin{pmatrix}  &&&1 \\ 1&&& \\ &\ddots&& \\ &&1&\end{pmatrix}\in \mathbb{R}^{(p-1)p^{d-1}\times (p-1)p^{d-1}}$ be the shift matrix on blocks of size $(p-1)p^{d-1}$, and define $\sigma_k(\boldsymbol{u})=\begin{pmatrix}  \boldsymbol{u}_1 \\  \vdots \\\boldsymbol{\sigma}\boldsymbol{u}_k\\ \vdots \\ \boldsymbol{u}_m\end{pmatrix}$, then $\boldsymbol{\sigma}\boldsymbol{A}_{ij}=\boldsymbol{A}_{ij}\boldsymbol{\sigma}$ as $\boldsymbol{A}_{ij}$ is circulant. Furthermore, when $i\neq j$, then $\boldsymbol{\sigma}\boldsymbol{A}_{ij}=\boldsymbol{A}_{ij}\boldsymbol{\sigma}=\boldsymbol{A}_{ij}$ as $\boldsymbol{B}_{ij}=p^{-|i-j|}\boldsymbol{1}\boldsymbol{1}^\top$ when $i\neq j$. Therefore, we have
\begin{align}
\label{eq:sigma_commute_Aii}
    \boldsymbol{\sigma}\boldsymbol{A}_{ii}=\boldsymbol{A}_{ii}\boldsymbol{\sigma}
\end{align}
and
\begin{align}
\label{eq:sigma_commute_Aij}
\boldsymbol{\sigma}\boldsymbol{A}_{ij}=\boldsymbol{A}_{ij}\boldsymbol{\sigma}=\boldsymbol{A}_{ij}
\end{align}
when $i\neq j$.

\begin{remark}
$\boldsymbol{A}$ is symmetric$\iff\boldsymbol{A}=\boldsymbol{A}^\top$;
$\boldsymbol{A}$ is persymmetric$\iff\boldsymbol{A}=\boldsymbol{J}\boldsymbol{A}^\top\boldsymbol{J}$;
$\boldsymbol{A}$ is centrosymmetric$\iff\boldsymbol{A}=\boldsymbol{J}\boldsymbol{A}\boldsymbol{J}$. And $\boldsymbol{A}$ is called bisymmetric if it is both symmetric and persymmetric. Among the three, knowing two allows one to deduce the third. 
\end{remark}

\begin{lemma}
\label{lem:cyclic_symmetric}
If $\boldsymbol{u}$ solves $\boldsymbol{A}\boldsymbol{u}+\rho e^{\boldsymbol{u}}=\boldsymbol{b}$, then $\sigma_k(\boldsymbol{u})$ solves $\boldsymbol{A}\boldsymbol{v}+\rho e^{\boldsymbol{v}}=\sigma_k(\boldsymbol{b})$.
\end{lemma}
\begin{proof}
As $\boldsymbol{A}\boldsymbol{u}+\rho e^{\boldsymbol{u}}=\boldsymbol{b}$, so we have
\begin{align*}
\begin{cases}
\displaystyle\sum_j\boldsymbol{A}_{1j}\boldsymbol{u}_j+\rho e^{\boldsymbol{u}_1}=\boldsymbol{b}_1,\\
\qquad\qquad\quad\vdots\\
\displaystyle\sum_j\boldsymbol{A}_{kj}\boldsymbol{u}_j+\rho e^{\boldsymbol{u}_k}=\boldsymbol{b}_k,\\
\qquad\qquad\quad\vdots\\
\displaystyle\sum_j\boldsymbol{A}_{mj}\boldsymbol{u}_j+\rho e^{\boldsymbol{u}_m}=\boldsymbol{b}_m.
\end{cases}
\end{align*}
Multiplying the $k$-th block by $\boldsymbol{\sigma}$ gives
\begin{align*}
\begin{cases}
\displaystyle\sum_j\boldsymbol{A}_{1j}\boldsymbol{u}_j+\rho e^{\boldsymbol{u}_1}=\boldsymbol{b}_1,\\
\qquad\qquad\quad\vdots\\
\displaystyle\sum_j\boldsymbol{\sigma}\boldsymbol{A}_{kj}\boldsymbol{u}_j+\rho \boldsymbol{\sigma}e^{\boldsymbol{u}_k}=\boldsymbol{\sigma}\boldsymbol{b}_k,\\
\qquad\qquad\quad\vdots\\
\displaystyle\sum_j\boldsymbol{A}_{mj}\boldsymbol{u}_j+\rho e^{\boldsymbol{u}_m}=\boldsymbol{b}_m,
\end{cases}
\end{align*}
which is equivalent to
\begin{align*}
\begin{cases}
\displaystyle\sum_{j\neq k}\boldsymbol{A}_{1j}\boldsymbol{u}_j+\boldsymbol{A}_{1k}\boldsymbol{\sigma}\boldsymbol{u}_k+\rho e^{\boldsymbol{u}_1}=\boldsymbol{b}_1,\\
\qquad\qquad\qquad\qquad\vdots\\
\displaystyle\sum_{j\neq k}\boldsymbol{A}_{kj}\boldsymbol{u}_j+\boldsymbol{A}_{kk}\boldsymbol{\sigma}\boldsymbol{u}_k+\rho e^{\boldsymbol{\sigma}\boldsymbol{u}_k}=\boldsymbol{\sigma}\boldsymbol{b}_k,\\
\qquad\qquad\qquad\qquad\vdots\\
\displaystyle\sum_{j\neq k}\boldsymbol{A}_{mj}\boldsymbol{u}_j+\boldsymbol{A}_{mk}\boldsymbol{\sigma}\boldsymbol{u}_k+\rho e^{\boldsymbol{u}_m}=\boldsymbol{b}_m.
\end{cases}
\end{align*}
By \eqref{eq:sigma_commute_Aii} and \eqref{eq:sigma_commute_Aij}. Thus $\sigma_k(\boldsymbol{u})$ is a solution of $\boldsymbol{A}\boldsymbol{v}+\rho e^{\boldsymbol{v}}=\sigma_k(\boldsymbol{b})$.
\end{proof}

Lemma \ref{lem:central_symmetric} and Lemma \ref{lem:cyclic_symmetric} imply certain symmetric properties: since $\boldsymbol{b}=p^d\boldsymbol{e}_y$, so given one solution, we can construct many solutions by folding or rotating the solution unless the solution is radially symmetric. And by the following theorem we proved that there only exist radially symmetric solutions when $\rho$ is small.

\begin{theorem}
\label{thm:radially_symmetric_small_rho}
When $\rho\leq p^{-d}-p^{-d-m}$, denote $u_y(x)$ be the solution of \eqref{eq:matrix_equation}, then $u_y(x)=u(\text{d}(x,y),|x|_p,|y|_p)$, which means that the solution is radially symmetric. Moreover, if $\text{d}(x_1,y_1)=\text{d}(x_2,y_2)$ and $v_p(x_1)+v_p(x_2)=v_p(y_1)+v_p(y_2)=m-1$, then $u_{y_1}(x_1)=u_{y_2}(x_2)$.
\end{theorem}
\begin{proof}
From \eqref{eq:compute_Du}, we compute
\begin{equation}
\begin{aligned}
\label{eq:compute_p^dDu}
p^dDu(x)=&\sum_{z\neq x}\frac{|z|_{p}|x|_{p}}{|z-x|_{p}^2}(u(z)-u(x))\\
=&\sum_{z\neq x}\frac{|z|_{p}|x|_{p}}{|z-x|_{p}^2}u(z)-\sum_{z\neq x}\frac{|z|_{p}|x|_{p}}{|z-x|_{p}^2}u(x)\\
=&\sum_{\substack{z\\0\leq v_p(z)\leq v_p(x)-1}}\frac{|x|_p}{|z|_p}u(z)+\sum_{\substack{z\\v_p(x)+1\leq v_p(z)\leq m-1}}\frac{|z|_p}{|x|_p}u(z)+\sum_{j=0}^{d-1}\sum_{\substack{a_i=b_i\\0\leq i\leq j-1\\a_{j}\neq b_{j}\\a_{j+1}\dots a_{d-1}}}p^{2j}u(z)\\
&-\left(\sum_{\substack{z\\0\leq v_p(z)\leq v_p(x)-1}}\frac{|x|_p}{|z|_p}+\sum_{\substack{z\\v_p(x)+1\leq v_p(z)\leq m-1}}\frac{|z|_p}{|x|_p}+\sum_{j=0}^{d-1}\sum_{\substack{a_i=b_i\\0\leq i\leq j-1\\a_{j}\neq b_{j}\\a_{j+1}\dots a_{d-1}}}p^{2j}\right)u(x).
\end{aligned}
\end{equation}

Since $a_0\neq b_0$ correspond to $(p-2)p^{d-1}$ cases and $a_i=b_i,\ i=0,1,\dots,j-1,\ a_j\neq b_j$ has $(p-1)p^{d-1-j}$ cases for $1\leq j\leq d-1$, so the coefficient of $u(x)$ is
\begin{equation}
\begin{aligned}
\label{eq:coefficient_of_u(x)}
&\sum_{\substack{z\\0\leq v_p(z)\leq v_p(x)-1}}\frac{|x|_p}{|z|_p}+\sum_{\substack{z\\v_p(x)+1\leq v_p(z)\leq m-1}}\frac{|z|_p}{|x|_p}+\sum_{j=0}^{d-1}\sum_{\substack{a_i=b_i\\0\leq i\leq j-1\\a_{j}\neq b_{j}\\a_{j+1}\dots a_{d-1}}}p^{2j}\\
=&(p-1)p^{d-1}\sum_{v_p(z)\neq v_p(x)}p^{-|v_p(z)-v_p(x)|}+(p-2)p^{d-1}+(p-1)p^{d-1-j}\sum_{j=1}^{d-1}p^{2j}\\
=&(p-1)p^{d-1}\sum_{v_p(z)\neq v_p(x)}p^{-|v_p(z)-v_p(x)|}+(p-2)p^{d-1}+(p-1)p^{d}\sum_{j=0}^{d-2}p^{j}\\
=&(p-1)p^{d-1}\frac{2-|x|_p-p^{1-m}|x|_p^{-1}}{p-1}+(p-2)p^{d-1}+p^{d}(p^{d-1}-1)\\
=&(2-|x|_p-p^{1-m}|x|_p^{-1})p^{d-1}+p^{2d-1}-2p^{d-1}\\
=&p^{2d-1}-(|x|_p+p^{1-m}|x|_p^{-1})p^{d-1},
\end{aligned}
\end{equation}
which coincides with the negative of the diagonal entries of $\boldsymbol{B}$. 

Note that
\begin{align*}
&\sum_{j=0}^{d-1}\sum_{\substack{a_i=b_i\\0\leq i\leq j-1\\a_{j}\neq b_{j}\\a_{j+1}\dots a_{d-1}}}p^{2j}u(z)+p^{2(d-1)}u(x)\\
=&\sum_{\substack{a_{0}\neq b_{0}\\a_{1}\dots a_{d-1}}}u(z)+\sum_{\substack{a_0=b_0\\a_{1}\neq b_{1}\\a_{2}\dots a_{d-1}}}p^{2}u(z)+\dots+\sum_{\substack{a_0=b_0\\\vdots\\a_{d-2}=b_{d-2}\\a_{d-1}}}p^{2(d-1)}u(z)\\
=&\sum_{\substack{z\\v_p(z)=v_p(x)}}u(z)+\sum_{\substack{a_0=b_0\\a_{1}\neq b_{1}\\a_{2}\dots a_{d-1}}}(p^{2}-1)u(z)+\dots+\sum_{\substack{a_0=b_0\\\vdots\\a_{d-2}=b_{d-2}\\a_{d-1}}}(p^{2(d-1)}-1)u(z)\\
=&\sum_{\substack{z\\v_p(z)=v_p(x)}}u(z)+\sum_{\substack{a_0=b_0\\a_{1}\dots a_{d-1}}}(p^{2}-1)u(z)+\sum_{\substack{a_0=b_0\\a_1=b_1\\a_{2}\neq b_{2}\\a_{3}\dots a_{d-1}}}(p^{4}-p^2)u(z)+\dots+\sum_{\substack{a_0=b_0\\\vdots\\a_{d-2}=b_{d-2}\\a_{d-1}}}(p^{2(d-1)}-p^2)u(z)\\
=&\cdots\\
=&\sum_{\substack{z\\v_p(z)=v_p(x)}}u(z)+\sum_{\substack{a_0=b_0\\a_{1}\dots a_{d-1}}}(p^{2}-1)u(z)+\dots+\sum_{\substack{a_0=b_0\\\vdots\\a_{d-2}=b_{d-2}\\a_{d-1}}}(p^{2(d-1)}-p^{2(d-2)})u(z)\\
=&\sum_{\substack{z\\v_p(z)=v_p(x)}}u(z)+\sum_{j=0}^{d-2}\sum_{\substack{a_i=b_i\\0\leq i\leq j\\a_{j+1}\dots a_{d-1}}}(p^{2(j+1)}-p^{2j})u(z),
\end{align*}
so we have
\begin{align*}
p^dDu(x)=&\sum_{\substack{z\\0\leq v_p(z)\leq v_p(x)-1}}\frac{|x|_p}{|z|_p}u(z)+\sum_{\substack{z\\v_p(x)+1\leq v_p(z)\leq m-1}}\frac{|z|_p}{|x|_p}u(z)+\sum_{\substack{z\\v_p(z)=v_p(x)}}u(z)\\
&+\sum_{j=0}^{d-2}\sum_{\substack{a_i=b_i\\0\leq i\leq j\\a_{j+1}\dots a_{d-1}}}(p^{2(j+1)}-p^{2j})u(z)-(p^{2d-1}+p^{2d-2}-(|x|_p+p^{1-m}|x|_p^{-1})p^{d-1})u(x),
\end{align*}
then when $x\neq y$ the equation \eqref{eq:mean_field_equation_finite_quotient} becomes
\begin{equation}
\begin{aligned}
\label{eq:equation_for_x_neq_y}
&\rho p^de^{u(x)}+\sum_{\substack{z\\0\leq v_p(z)\leq v_p(x)-1}}\frac{|x|_p}{|z|_p}u(z)+\sum_{\substack{z\\v_p(x)+1\leq v_p(z)\leq m-1}}\frac{|z|_p}{|x|_p}u(z)+\sum_{\substack{z\\v_p(z)=v_p(x)}}u(z)\\
&+\sum_{j=0}^{d-2}\sum_{\substack{a_i=b_i\\0\leq i\leq j\\a_{j+1}\dots a_{d-1}}}(p^{2(j+1)}-p^{2j})u(z)-(p^{2d-1}+p^{2d-2}-(|x|_p+p^{1-m}|x|_p^{-1})p^{d-1})u(x)=0.
\end{aligned}
\end{equation}
Denoting $u(x)=u_{l}(b_0,b_1,\dots,b_{d-1})$, by fixing $b_{0},b_{1},\dots,b_{d-2}$ and varying $b_{d-1}$, we see that $u(x)$ satisfies \eqref{eq:equation_for_x_neq_y}. In addition, the function $f(x)=\rho p^de^x-kx$ for $\ x\leq d\ln p$ is monotonically increasing when $p^{2d}\rho\leq k$. Then by Lemma \ref{lem:rough_upperbound}, if $p^{2d}\rho\leq p^{2d-1}+p^{2d-2}-(p^{-l}+p^{l+1-m})p^{d-1}$, $u_{l}(b_0,b_1,\dots,b_{d-1})$ is constant for all $b_{d-1}$ except at $x=y$. Hence for $x\neq y$, we denote $u(x)$ by $u_{l}(b_0,b_1,\dots,b_{d-2})$.

When $v_p(x)=l\neq v_p(y)$, \eqref{eq:equation_for_x_neq_y} becomes
\begin{align*}
&\sum_{\substack{z\\0\leq k\leq l-1}}p^{k-l}u(z)+\sum_{\substack{z\\l+1\leq k\leq m-1}}p^{l-k}u(z)+\sum_{\substack{z\\v_p(z)=l}}u(z)\\
&+\sum_{j=0}^{d-3}\sum_{\substack{a_i=b_i\\0\leq i\leq j\\a_{j+1}\dots a_{d-1}}}(p^{2(j+1)}-p^{2j})u_{l}(a_0,a_1,\dots,a_{d-2})+p(p^{2(d-1)}-p^{2(d-2)})u_{l}(b_0,b_1,\dots,b_{d-2})\\
&+\rho p^{d} e^{u_{l}(b_0,b_1,\dots,b_{d-2})}-(p^{2d-1}+p^{2d-2}-(p^{-l}+p^{l+1-m})p^{d-1})u_{l}(b_0,b_1,\dots,b_{d-2})=0.
\end{align*}
This simplifies to
\begin{equation}
\begin{aligned}
\label{eq:equation_for_vpx_neq_vpy_d-2}
&\sum_{\substack{z\\0\leq k\leq l-1}}p^{k-l}u(z)+\sum_{\substack{z\\l+1\leq k\leq m-1}}p^{l-k}u(z)+\sum_{\substack{z\\v_p(z)=l}}u(z)+\sum_{j=0}^{d-3}\sum_{\substack{a_i=b_i\\0\leq i\leq j\\a_{j+1}\dots a_{d-1}}}(p^{2(j+1)}-p^{2j})u_{l}(a_0,a_1,\dots,a_{d-2})\\
&+\rho p^{d}e^{u_{l}(b_0,b_1,\dots,b_{d-2})}-(p^{2d-2}+p^{2d-3}-(p^{-l}+p^{l+1-m})p^{d-1})u_{l}(b_0,b_1,\dots,b_{d-2})=0.
\end{aligned}
\end{equation}
Fixing $b_{0},b_{1},\dots,b_{d-3}$ and varying $b_{d-2}$, we see that $u(x)$ satisfies \eqref{eq:equation_for_vpx_neq_vpy_d-2}. So when $\rho p^{2d}\leq p^{2d-2}+p^{2d-3}-(p^{-l}+p^{l+1-m})p^{d-1}$, $u_{l}(b_0,b_1,\dots,b_{d-2})$ is constant for all $b_{d-2}$, denote it by $u_{l}(b_0,b_1,\dots,b_{d-3})$. Similarly, if we fix $b_{0},b_{1},\dots,b_{d-4}$ and when $\rho p^{2d}\leq p^{2d-2}+p^{2d-3}-(p^{-l}+p^{l+1-m})p^{d-1}$, $u_{l}(b_0,b_1,\dots,b_{d-3})$ is constant for all $b_{d-3}$. Proceeding recursively, we obtain
\begin{align*}
&\sum_{\substack{z\\0\leq k\leq l-1}}p^{k-l}u(z)+\sum_{\substack{z\\l+1\leq k\leq m-1}}p^{l-k}u(z)+\sum_{\substack{z\\v_p(z)=l}}u(z)+\sum_{j=0}^{s}\sum_{\substack{a_i=b_i\\0\leq i\leq j\\a_{j+1}\dots a_{d-1}}}(p^{2(j+1)}-p^{2j})u_{l}(a_0,\dots,a_{s},a_{s+1})\\
&+\rho p^{d}e^{u_{l}(b_0,\dots,b_{s},b_{s+1})}-(p^{d+1+s}+p^{2d+s}-(p^{-l}+p^{l+1-m})p^{d-1})u_{l}(b_0,\dots,b_{s},b_{s+1})=0,
\end{align*}
$\forall s=0,1,\dots,d-2$. So when $\rho p^{2d}\leq p^{d+1+s}+p^{d+s}-(p^{-l}+p^{l+1-m})p^{d-1}$, $u_{l}$ dependent only on $b_0,b_1,\dots,b_{s}$, $\forall s=0,1,\dots,d-2$. In conclusion, when $v_p(x)=l\neq v_p(y)$, if $\rho p^{2d}\leq p^{d}+p^{d-1}-(p^{-l}+p^{l+1-m})p^{d-1}$, then $u_{l}$ is constant.

When $v_p(x)=l=v_p(y)$, denote $y=p^l(c_0+c_1p+\dots+c_{d-1}p^{d-1})$, since $u_{l}(b_0,b_1,\dots,b_{d-1})$ is constant for all $b_{d-1}$ except at $x=y$, fixing $b_i=c_i$, $0\leq i\leq d-2$, $b_{d-1}\neq c_{d-1}$, then $u_l(c_0,\dots,c_{d-2},b_{d-1})$ is independent of $b_{d-1}$, we denote this invariant value by $u_l(c_0,\dots,c_{d-2})$.

If we fix $b_i=c_i$, $0\leq i\leq s$, $b_{s+1}\neq c_{s+1}$ and suppose $u_l(c_0,\dots,c_{s},b_{s+1},\dots,b_{d-1})$ is independent of $b_{s+2},\dots,b_{d-1}$, denote it by $u_l(c_0,\dots,c_{s},b_{s+1})$. Next, we prove that $u_l(c_0,\dots,c_{s},b_{s+1})$ is independent of $b_{s+1}$: From \eqref{eq:compute_p^dDu} and \eqref{eq:coefficient_of_u(x)}, we can get
\begin{equation}
\begin{aligned}
\label{eq:compute_p^dDu_index}
p^{d}Du(x)=&\sum_{\substack{z\\0\leq k\leq l-1}}p^{k-l}u(z)+\sum_{\substack{z\\l+1\leq k\leq m-1}}p^{l-k}u(z)+\sum_{j=0}^{d-1}\sum_{\substack{a_i=b_i\\0\leq i\leq j-1\\a_{j}\neq b_{j}\\a_{j+1}\dots a_{d-1}}}p^{2j}u(z)\\
&-(p^{2d-1}-(p^{-l}+p^{l+1-m})p^{d-1})u_l(c_0,\dots,c_{s},b_{s+1})\\
=&\sum_{\substack{z\\0\leq k\leq l-1}}p^{k-l}u(z)+\sum_{\substack{z\\l+1\leq k\leq m-1}}p^{l-k}u(z)+\sum_{j=0}^{s}\sum_{\substack{a_i=c_i\\0\leq i\leq j-1\\a_{j}\neq c_{j}\\a_{j+1}\dots a_{d-1}}}p^{2j}u(z)+\sum_{\substack{a_i=c_i\\0\leq i\leq s\\a_{s+1}\neq b_{s+1}\\a_{s+2}\dots a_{d-1}}}p^{2(s+1)}u(z)\\
&+\sum_{j=s+2}^{d-1}\sum_{\substack{a_i=b_i\\0\leq i\leq j-1\\a_{j}\neq b_{j}\\a_{j+1}\dots a_{d-1}}}p^{2j}u(z)-(p^{2d-1}-(p^{-l}+p^{l+1-m})p^{d-1})u_l(c_0,\dots,c_{s},b_{s+1}).
\end{aligned}
\end{equation}
We first compute
\begin{align*}
&\sum_{\substack{a_i=c_i\\0\leq i\leq s\\a_{s+1}\neq b_{s+1}\\a_{s+2}\dots a_{d-1}}}p^{2(s+1)}u(z)\\
=&p^{2(s+1)}\Biggl(p^{d-s-2}\sum_{\substack{a_{s+1}\neq b_{s+1}\\a_{s+1}\neq c_{s+1}}}u_l(c_0,\dots,c_{s},a_{s+1})\\
&+(p-1)p^{d-s-3}u_l(c_0,\dots,c_{s+1})+\dots+(p-1)u_l(c_0,\dots,c_{d-2})+u(y)\Biggr)\\
=&p^{d+s}\sum_{\substack{a_{s+1}\neq b_{s+1}\\a_{s+1}\neq c_{s+1}}}u_l(c_0,\dots,c_{s},a_{s+1})+\sum_{j=s+1}^{d-2}(p-1)p^{d+2s-j}u_l(c_0,\dots,c_{j})+p^{2(s+1)}u(y)
\end{align*}
and
\begin{align*}
\sum_{j=s+2}^{d-1}\sum_{\substack{a_i=b_i\\0\leq i\leq j-1\\a_{j}\neq b_{j}\\a_{j+1}\dots a_{d-1}}}p^{2j}u(z)=&\sum_{j=s+2}^{d-1}(p-1)p^{d-1-j}p^{2j}u_l(c_0,\dots,c_{s},b_{s+1})\\
=&(p^{2d-1}-p^{d+s+1})u_l(c_0,\dots,c_{s},b_{s+1}).
\end{align*}
Substituting into \eqref{eq:compute_p^dDu_index} yields
\begin{equation}
\begin{aligned}
\label{eq:compute_p^dDu_index_simplified}
&p^{d}Du(x)\\
=&\sum_{\substack{z\\0\leq k\leq l-1}}p^{k-l}u(z)+\sum_{\substack{z\\l+1\leq k\leq m-1}}p^{l-k}u(z)+\sum_{j=0}^{s}\sum_{\substack{a_i=c_i\\0\leq i\leq j-1\\a_{j}\neq c_{j}\\a_{j+1}\dots a_{d-1}}}p^{2j}u(z)\\
&+p^{d+s}\sum_{\substack{a_{s+1}\neq b_{s+1}\\a_{s+1}\neq c_{s+1}}}u_l(c_0,\dots,c_{s},a_{s+1})+\sum_{j=s+1}^{d-2}(p-1)p^{d+2s-j}u_l(c_0,\dots,c_{j})+p^{2(s+1)}u(y)\\
&-(p^{d+s+1}-(p^{-l}+p^{l+1-m})p^{d-1})u_l(c_0,\dots,c_{s},b_{s+1})\\
=&\sum_{\substack{z\\0\leq k\leq l-1}}p^{k-l}u(z)+\sum_{\substack{z\\l+1\leq k\leq m-1}}p^{l-k}u(z)+\sum_{j=0}^{s}\sum_{\substack{a_i=c_i\\0\leq i\leq j-1\\a_{j}\neq c_{j}\\a_{j+1}\dots a_{d-1}}}p^{2j}u(z)\\
&+p^{d+s}\sum_{\substack{a_{s+1}\neq c_{s+1}}}u_l(c_0,\dots,c_{s},a_{s+1})+\sum_{j=s+1}^{d-2}(p-1)p^{d+2s-j}u_l(c_0,\dots,c_{j})+p^{2(s+1)}u(y)\\
&-(p^{d+s+1}+p^{d+s}-(p^{-l}+p^{l+1-m})p^{d-1})u_l(c_0,\dots,c_{s},b_{s+1}).
\end{aligned}
\end{equation}
Then equation \eqref{eq:equation_for_x_neq_y} becomes
\begin{align*}
&\sum_{\substack{z\\0\leq k\leq l-1}}p^{k-l}u(z)+\sum_{\substack{z\\l+1\leq k\leq m-1}}p^{l-k}u(z)+\sum_{j=0}^{s}\sum_{\substack{a_i=c_i\\0\leq i\leq j-1\\a_{j}\neq c_{j}\\a_{j+1}\dots a_{d-1}}}p^{2j}u(z)\\
&+p^{d+s}\sum_{\substack{a_{s+1}\neq c_{s+1}}}u_l(c_0,\dots,c_{s},a_{s+1})+\sum_{j=s+1}^{d-2}(p-1)p^{d+2s-j}u_l(c_0,\dots,c_{j})+p^{2(s+1)}u(y)\\
&-(p^{d+s+1}+p^{d+s}-(p^{-l}+p^{l+1-m})p^{d-1})u_l(c_0,\dots,c_{s},b_{s+1})+\rho p^{d} e^{u_{l}(c_0,\dots,c_s,b_{s+1})}=0.
\end{align*}
So when $p^{2d}\rho\leq p^{d+1+s}+p^{d+s}-(p^{-l}+p^{l+1-m})p^{d-1}$, $u_{l}$ depend only on $c_0,c_1,\dots,c_{s}$, $\forall s=0,1,\dots,d-2$. Then we get the result by induction. In conclusion, when $v_p(x)=l=v_p(y)$, $p^{2d}\rho\leq p^{d+1+s}+p^{d+s}-(p^{-l}+p^{l+1-m})p^{d-1}$, $u_{l}(c_0,\dots,c_{s})$ is constant.
Then when
\begin{align*}
\rho&\leq p^{-2d}\min_{\substack{-1\leq s\leq d-2\\0\leq l\leq m-1}}\{p^{d+1+s}+p^{d+s}-(p^{-l}+p^{l+1-m})p^{d-1}\}\\
&=p^{-2d}(p^{d}+p^{d-1}-(1+p^{1-m})p^{d-1})\\
&=p^{-2d}(p^{d}-p^{d-m})\\
&=p^{-d}-p^{-d-m},
\end{align*}
$u(x)$ is radially symmetric. By Lemma \ref{lem:central_symmetric} and Lemma \ref{lem:cyclic_symmetric}, we can see that $u_y(x)=u(\text{d}(x,y),|x|_p,|y|_p)$, and if $\text{d}(x_1,y_1)=\text{d}(x_2,y_2)$ with $v_p(x_1)+v_p(x_2)=v_p(y_1)+v_p(y_2)=m-1$, then $u_{y_1}(x_1)=u_{y_2}(x_2)$.
\end{proof}

Numerical experiments reveal that these radially symmetric solutions exhibit strict monotonicity. $u(x)$ increases with the distance from the singular point $y$.

\begin{theorem}
    \label{thm:monoton_small_rho}
    Assume $\rho\leq p^{-d}-p^{-d-m}$ and $v_p(x)=l=v_p(y)$, then the radially symmetric solution satisfies $u_l(c_0,\dots,c_{d-1})<u_l(c_0,\dots,c_{d-2})<\dots<u_l(c_0)<u_l$
\end{theorem}
\begin{proof}
From Theorem \ref{thm:radially_symmetric_small_rho}, when $\rho\leq p^{-d}-p^{-d-m}$, the solution is radially symmetric, so the equation variables can be reduced from $m(p-1)p^{d-1}$ to $p+m$. Then denote these variables by $u_0,\dots,u_{l-1},u_{l},u_{l}(c_0),\dots,u_l(c_0,\dots,c_{d-1}),u_{l+1},\dots,u_{m-1}$.

When $x\neq y$, by letting $v_p(z)=k$, equation \eqref{eq:matrix_equation} becomes
\begin{equation}
\begin{aligned}
\label{eq:equation_radial_symmetric_x_neq_y}
&\sum_{\substack{z\\0\leq k\leq l-1}}p^{k-l}u(z)+\sum_{\substack{z\\l+1\leq k\leq m-1}}p^{l-k}u(z)+\sum_{j=0}^{d-1}\sum_{\substack{a_i=c_i\\0\leq i\leq j-1\\a_{j}\neq c_{j}\\a_{j+1}\dots a_{d-1}}}p^{2j}u(z)\\
&-(p^{2d-1}-(p^{-l}+p^{l+1-m})p^{d-1})u_{l}+p^{d}\rho e^{u_{l}(x)}=0.
\end{aligned}
\end{equation}
When $b_i=c_i$ for $0\leq i\leq s$, then from \eqref{eq:compute_p^dDu_index_simplified}, we have
\begin{equation}
\begin{aligned}
\label{eq:equation_radial_symmetric_x_neq_y_crutial_part}
&\sum_{j=0}^{d-1}\sum_{\substack{a_i=c_i\\0\leq i\leq j-1\\a_{j}\neq c_{j}\\a_{j+1}\dots a_{d-1}}}p^{2j}u(z)-(p^{2d-1}-(p^{-l}+p^{l+1-m})p^{d-1})u_l(c_0,\dots,c_{s})\\
=&\sum_{j=0}^{s}\sum_{a_j\neq c_j}p^{d-1-j}p^{2j}u_l(c_0,\dots,c_{j-1})+p^{d+s}\sum_{a_{s+1}\neq c_{s+1}}u_l(c_0,\dots,c_{s})+\sum_{j=s+1}^{d-2}(p-1)p^{d+2s-j}u_l(c_0,\dots,c_{j})\\
&+p^{2(s+1)}u_l(c_0,\dots,c_{d-1})-(p^{s+2}+p^{s+1}-p^{-l}-p^{l+1-m})p^{d-1}u_l(c_0,\dots,c_{s})\\
=&\sum_{j=0}^{s}\sum_{a_j\neq c_j}p^{d-1+j}u_l(c_0,\dots,c_{j-1})+p^{d+s}(p-1-\delta_{s,-1})u_l(c_0,\dots,c_{s})\\
&+p^{2(s+1)}\left(\sum_{j=s+1}^{d-2}(p-1)p^{d-2-j}u_l(c_0,\dots,c_{j})+u_l(c_0,\dots,c_{d-1})\right)\\
&-(p^{s+2}+p^{s+1}-p^{-l}-p^{l+1-m})p^{d-1}u_l(c_0,\dots,c_{s})\\
=&\sum_{j=0}^{s}(p-1-\delta_{j0})p^{d-1+j}u_l(c_0,\dots,c_{j-1})+p^{2(s+1)}\left(\sum_{j=s+1}^{d-2}(p-1)p^{d-2-j}u_l(c_0,\dots,c_{j})+u_l(c_0,\dots,c_{d-1})\right)\\
&-(2p^{s+1}+\delta_{s,-1}-p^{-l}-p^{l+1-m})p^{d-1}u_l(c_0,\dots,c_{s})\\
\end{aligned}
\end{equation}
where $s=-1,0,1,\dots,d-2$.

We proceed by backward mathematical induction to prove this theorem:

From Lemma \ref{lem:u_y_smallest}, $u(y)$ is the unique global minimum of the solution, so we have $u_l(c_0,\dots,c_{d-2})>u_l(c_0,\dots,c_{d-1})$.

Suppose $u_l(c_0,\dots,c_{s+1})>\dots>u_l(c_0,\dots,c_{d-1})$, now we want to compare $u_l(c_0,\dots,c_{s+1})$ with $u_l(c_0,\dots,c_{s})$. First, by \eqref{eq:equation_radial_symmetric_x_neq_y} and \eqref{eq:equation_radial_symmetric_x_neq_y_crutial_part}, they satisfy the following equations:
\begin{equation}
    \begin{aligned}
    \label{eq:equation_radial_symmetric_up_to_s}
        &(p-1)p^{d-1}\left(\sum_{k=0}^{l-1}p^{k-l}u_k+\sum_{k=l+1}^{m-1}p^{l-k}u_k\right)
        +\sum_{j=0}^{s}(p-1-\delta_{j0})p^{d-1+j}u_l(c_0,\dots,c_{j-1})\\
        &+p^{2(s+1)}\left(\sum_{j=s+1}^{d-2}(p-1)p^{d-2-j}u_l(c_0,\dots,c_{j})+u_l(c_0,\dots,c_{d-1})\right)\\
        &-(2p^{s+1}+\delta_{s,-1}-p^{-l}-p^{l+1-m})p^{d-1}u_l(c_0,\dots,c_{s})+p^{d}\rho e^{u_{l}(c_0,\dots,c_s)}=0
    \end{aligned}
\end{equation}
and
\begin{equation}
    \begin{aligned}
    \label{eq:equation_radial_symmetric_up_to_s+1}
        &(p-1)p^{d-1}\left(\sum_{k=0}^{l-1}p^{k-l}u_k+\sum_{k=l+1}^{m-1}p^{l-k}u_k\right)
        +\sum_{j=0}^{s+1}(p-1-\delta_{j0})p^{d-1+j}u_l(c_0,\dots,c_{j-1})\\
        &+p^{2(s+2)}\left(\sum_{j=s+2}^{d-2}(p-1)p^{d-2-j}u_l(c_0,\dots,c_{j})+u_l(c_0,\dots,c_{d-1})\right)\\
        &-(2p^{s+2}-p^{-l}-p^{l+1-m})p^{d-1}u_l(c_0,\dots,c_{s+1})+p^{d}\rho e^{u_{l}(c_0,\dots,c_{s+1})}=0.
    \end{aligned}
\end{equation}
Subtracting \eqref{eq:equation_radial_symmetric_up_to_s+1} from \eqref{eq:equation_radial_symmetric_up_to_s} gives
    \begin{align*}
        0=&p^{d}\rho e^{u_{l}(c_0,\dots,c_{s+1})}-p^{d}\rho e^{u_{l}(c_0,\dots,c_{s})}\\
        &-(2p^{s+2}-p^{-l}-p^{l+1-m})p^{d-1}u_l(c_0,\dots,c_{s+1})-p^{2(s+1)}(p-1)p^{d-2-s-1}u_{l}(c_0,\dots,c_{s+1})\\
        &+(p^2-1)p^{2(s+1)}(\sum_{j=s+2}^{d-2}(p-1)p^{d-2-j}u_l(c_0,\dots,c_{j})+u_l(c_0,\dots,c_{d-1}))\\
        &+(2p^{s+1}+\delta_{s,-1}-p^{-l}-p^{l+1-m})p^{d-1}u_l(c_0,\dots,c_{s})+(p-1-\delta_{s+1,0})p^{d+s}u_l(c_0,\dots,c_{s}),
    \end{align*}
which is
    \begin{align*}
        0=&p^{d}\rho e^{u_{l}(c_0,\dots,c_{s+1})}-p^{d}\rho e^{u_{l}(c_0,\dots,c_{s})}\\
        &-(2p^{s+2}+p^{s+1}-p^{s}-p^{-l}-p^{l+1-m})p^{d-1}u_l(c_0,\dots,c_{s+1})\\
        &+(p^2-1)p^{2(s+1)}\left(\sum_{j=s+2}^{d-2}(p-1)p^{d-2-j}u_l(c_0,\dots,c_{j})+u_l(c_0,\dots,c_{d-1})\right)\\
        &+p^{d-1}(p^{s+2}+p^{s+1}-p^{-l}-p^{l+1-m})u_l(c_0,\dots,c_{s}).
    \end{align*}
Then by induction hypothesis and noting that
    \begin{align*}
        &(p^2-1)p^{2(s+1)}\left(\sum_{j=s+2}^{d-2}(p-1)p^{d-2-j}+1\right)\\
        =&(p^2-1)p^{2(s+1)}\left((p-1)p^{d-2}\frac{p^{-s-1}-p^{-d+2}}{p-1}+1\right)\\
        =&p^{d-1}(p^{s+2}-p^s),
    \end{align*}
we obtain
    \begin{align*}
        0<&p^{d}\rho e^{u_{l}(c_0,\dots,c_{s+1})}-p^{d}\rho e^{u_{l}(c_0,\dots,c_{s})}\\
        &-(p^{s+2}+p^{s+1}-p^{-l}-p^{l+1-m})p^{d-1}u_l(c_0,\dots,c_{s+1})\\
        &+p^{d-1}(p^{s+2}+p^{s+1}-p^{-l}-p^{l+1-m})u_l(c_0,\dots,c_{s}),
    \end{align*}
which is
    \begin{align*}
        0<&\rho e^{u_{l}(c_0,\dots,c_{s+1})}-\rho e^{u_{l}(c_0,\dots,c_{s})}\\
        &-p^{-1}(p^{s+2}+p^{s+1}-p^{-l}-p^{l+1-m})u_l(c_0,\dots,c_{s+1})\\
        &+p^{-1}(p^{s+2}+p^{s+1}-p^{-l}-p^{l+1-m})u_l(c_0,\dots,c_{s}).
    \end{align*}
Then by the monotonicity of $f_s(x)=\rho e^x-p^{-1}(p^{s+2}+p^{s+1}-p^{-l}-p^{l+1-m})x$, if $\rho\leq p^{-d}-p^{-d-m}$, we have $u_{l}(c_0,\dots,c_{s+1})<u_{l}(c_0,\dots,c_{s})$.\\
\end{proof}
 
\begin{remark}
The differences $u_{l}(c_0,\dots,c_{s})-u_{l}(c_0,\dots,c_{s+1})$ decrease with $s$ as
    \begin{align*}
        &f_s(u_{l}(c_0,\dots,c_{s+1}))-f_s(u_{l}(c_0,\dots,c_{s}))\\
        =&(p^2-1)p^{2(s+1)}\left(\sum_{j=s+2}^{d-2}(p-1)p^{d-2-j}((u_{l}(c_0,\dots,c_{s+1})-u_l(c_0,\dots,c_{j})\right)\\
        &+(u_{l}(c_0,\dots,c_{s+1})-u_l(c_0,\dots,c_{d-1})),
    \end{align*}
so by the monotonicity of the solution and $f_s(x)$, the gap between them decreases.
\end{remark}

\begin{remark}
    There always exists radially symmetric solutions as we can view the $p+d$ points as a finite weighted graph. Then from the results in \cite{huang2020existence} we obtain the existence of the solution, regardless of the value of $\rho$.
\end{remark}

\begin{remark}
    The upper bound of $\rho$ that ensures $u$ is radially symmetric and monotone is highly dependent on the upper bound of $u$. And Lemma \ref{lem:rough_upperbound} yields only a crude estimate for the upper bound of $u$.
\end{remark}

From Theorem \ref{thm:monoton_small_rho}, we can provides a sharper estimate of the upper bound for the solution. We now refine Lemma \ref{lem:rough_upperbound}:

\begin{lemma}
\label{lem:refined_upperbound}
    If $\rho\leq p^{-d}-p^{-d-m}$, then the radially symmetric solution $u$ satisfies
    \begin{align*}
        u(i)<\begin{cases}
        \ln\left(\frac{p}{p-2}\right),\ &p>2;\\
        \ln 4,&p=2.
    \end{cases}
    \end{align*}    
\end{lemma}
\begin{proof}
    From Lemma \ref{lem:rough_upperbound}, we know that $\sum_{i}e^{u_i}=p^{d}$. Denoting $M$ to be the largest solution. Then from Theorem \ref{thm:monoton_small_rho}, we can see there at least $(p-2)p^{d-1}$ such $M$ when $p>2$, so we get
    \begin{align*}
        (p-2)p^{d-1}e^M<p^d
    \end{align*}
    thus $M<\ln\left(\frac{p}{p-2}\right)$.\\
    When $p=2$, there are at least $2^{d-2}$ such $M$, so $M<\ln 4$.
\end{proof}

Using this improved bound, we can refine Theorem \eqref{thm:radially_symmetric_small_rho} and Theorem \eqref{thm:monoton_small_rho} by method of continuity:

\begin{theorem}
\label{thm:refined_radial_symmetric_and_monotone}
     When $\rho\in(0,\rho^*)$, the radially symmetric solution is monotone with a universal upper bound $M$ independent of $d$ and $\rho$, where $\rho^*=\theta_{p,m} e^{-M}$ and $\theta_{p,m}$ is the smallest nonzero eigenvalue of $-\boldsymbol{A}$. Furthermore, when $\rho\in(0,(1-p^{-m})e^{-M}]$, the solution $u$ must be radially symmetric.
\end{theorem}

\begin{proof}
    Denoting
    \begin{align*}
        S=\{\rho\in(0,\rho^*)|\text{ The radially symmetric solution of \eqref{eq:matrix_equation} satisfies strict monotonicity}\},
    \end{align*}
     and
    \begin{align*}
        M=\sup_{u}\left\{\max_{x}u|\ u\text{ is the radially symmetric monotone solution of \eqref{eq:matrix_equation}}\right\}
    \end{align*}
    
    First, from Theorem \ref{thm:monoton_small_rho}, we have $(0,p^{-d}-p^{-d-m}]\in S$.

    Next, suppose $\rho_0\in S$, and the corresponding radially symmetric solution is $\boldsymbol{u}_0$. Then by the same argument of Lemma \ref{lem:refined_upperbound}, we have
    \begin{align*}
        u_0(i)\leq M<\begin{cases}
        \ln\left(\frac{p}{p-2}\right),\ &p>2;\\
        \ln 4,&p=2,
    \end{cases}
    \end{align*} 
    which means any radially symmetric solution have a universal upper bound, so $M$ is well defined. Now considering
    \begin{align*}
        F:\mathbb{R}\times V_{rad}&\to V_{rad}\\
        (\rho,\boldsymbol{u})&\mapsto F(\rho,\boldsymbol{u})=\boldsymbol{A}\boldsymbol{u}+\rho e^{\boldsymbol{u}}-\rho p^{d}\boldsymbol{e}_y,
    \end{align*}
    where $V_{rad}$ is a linear subspace formed by radially symmetric functions with dim$(V_{rad})\leq p+m$. Then
    \begin{align*}              L_{\rho_0}:=D_{\boldsymbol{u}}F(\rho_0,\boldsymbol{u}_0)=\boldsymbol{A}+\rho_0\text{diag}(e^{\boldsymbol{u}_0}):V_{rad}\to V_{rad}
    \end{align*}
    is nondegenerate since $\rho_0e^{u_0(i)}<\rho^*e^M=\theta_{p,m}$, then by the Implicit Function Theorem, there exists $\varepsilon>0$ such that
    \begin{align*}
        (\rho_0-\varepsilon,\rho_0+\varepsilon)&\to V_{rad}\\
        \rho&\mapsto \boldsymbol{u}_\rho
    \end{align*}
    satisfies $F(\rho,\boldsymbol{u}_\rho)=0$ and $\boldsymbol{u}_{\rho_0}=\boldsymbol{u}_0$. Since we can take $\varepsilon$ small enough to ensure $\boldsymbol{u}_\rho$ is radially symmetric, so $S$ is open in $(0,\rho^*)$.

    Last, suppose $\{\rho_n\}_n\subseteq S$ and $\rho_n\to \rho\in(0,\rho^*)$ with the corresponding radially symmetric solutions $\boldsymbol{u}_n$. Since dim$(V_{rad})\leq p+m$, then there exists a subsequence $\boldsymbol{u}_{n_k}\to\boldsymbol{u}$. By the continuity of \eqref{eq:matrix_equation}, $\boldsymbol{u}$ is the radially symmetric solution corresponding to $\rho$, and $\boldsymbol{u}$ is at least non-strictly monotone. Suppose there exists $s$ such that $u_l(c_0,\dots,c_{d-1})<u_l(c_0,\dots,c_{d-2})<\dots<u_l(c_0,\dots,c_{s+1})=u_l(c_0,\dots,c_{s})$, then by the same argument of Theorem \ref{thm:monoton_small_rho}, we have
    \begin{align*}
        f_s(u_l(c_0,\dots,c_{s+1}))>f_s(u_l(c_0,\dots,c_{s})),
    \end{align*}
    a contradiction. So $S$ is closed in $(0,\rho^*)$.

    In conclusion, $S=(0,\rho^*)$. Then by the same argument of Theorem \ref{thm:radially_symmetric_small_rho}, we obtain when $\rho\in(0,(1-p^{-m})e^{-M}]$, the solution $u$ must be radially symmetric.
\end{proof}

\subsection{Uniqueness of the solution}
In this part, we discuss the uniqueness of the solution. Numerical experiments indicate that for large $\rho$, solutions may not unique, and radial symmetry may be broken. Then by Lemma \ref{lem:central_symmetric} and Lemma \ref{lem:cyclic_symmetric}, we can construct many non-radially symmetric solutions. Besides, the monotonicity of the solutions may not be ensured for large $\rho$. For instance, if we take $p=3$, $m=1$, $d=2$, $\rho=5$ and $y=1$, then there exists non-monotone solution $\boldsymbol{u}_{nmon}$ and non-radially symmetric solution $\boldsymbol{u}_{nrad}$, see in Table \ref{tab:nmon_and_nrad_solution}.

\begin{table}[h!]
    \centering
    \begin{tabular}{|c|c|c|}
        \hline
        &$\boldsymbol{u}_{nmon}$& $\boldsymbol{u}_{nrad}$ \\ \hline
        $u(1)$ & -18.2689 & -22.2707\\ \hline
        $u(2)$ & -1.3831 & -0.4886\\ \hline
        $u(4)$ & 1.4168 & 1.9684\\ \hline
        $u(5)$ & -1.3831 & -0.4886\\ \hline
        $u(7)$ & 1.4168 & -8.7704\\ \hline
        $u(8)$ & -1.3831 & -0.4886\\ \hline
    \end{tabular}
    \caption{Non-monotone solutions and non-radially symmetric solutions}
    \label{tab:nmon_and_nrad_solution}
\end{table}

A natural question is whether uniqueness holds under the conditions of $\rho$ guaranteeing radially symmetric and monotonicity of the solution, the answer is affirmative. In order to investigate it, our first step is to study the eigenvalue of $\boldsymbol{A}$, i.e. the spectrum of $D$ on the finite quotient $\mathbb{Q}_{p}^{\times}/p^{m\mathbb{Z}}/(1+p^d\mathbb{Z}_{p})$. As we can pull back the functions from $\mathbb{Q}_{p}^{\times}/p^{m\mathbb{Z}}/(1+p^d\mathbb{Z}_{p})$ to $\mathbb{Q}_{p}^{\times}/p^{m\mathbb{Z}}$, so the eigenvalue of $\boldsymbol{A}$ is just the subset of the spectrum of $D$ by taking the first $d$-th eigenvalues from each component. Alternatively, we can compute the eigenvalues of $\boldsymbol{A}$ via linear algebra as follows:

\begin{theorem}
    Eigenvalues and the eigenvectors of $\boldsymbol{A}$.
\end{theorem}
\begin{proof}
    As $\boldsymbol{A}=p^{-d}\boldsymbol{B}$, we first consider $\boldsymbol{B}$. Denoting $n=(p-1)p^{d-1}$. Then $\boldsymbol{B}\in\mathbb{R}^{mn\times mn}$. From previous observation, $\boldsymbol{B}$ is a block matrix and each block is a circulant matrix. Furthermore, the off-diagonal part is a matrix with the same element. Denoting $\boldsymbol{P}\in\mathbb{R}^{m\times m}$ with $(\boldsymbol{P})_{ij}=p^{-|i-j|}$. And $\boldsymbol{T}\in\mathbb{R}^{n\times n}$ be the circulant part on diagonal with $(\boldsymbol{T})_{ij}=\left(\frac{|z|_p|x|_p}{|z-x|_p^2}-1\right)(1-\delta_{zx})$ (i.e. $\boldsymbol{T}$ is zero on diagonal and $\boldsymbol{T}$ is independent of $v_p(z)=v_p(x)$). Then we can write $\boldsymbol{B}$ as
    \begin{align*}        \boldsymbol{B}=\boldsymbol{P}\otimes\boldsymbol{1}_n\boldsymbol{1}_n^{\top}+\boldsymbol{I}_m\otimes\boldsymbol{T}-\text{diag}((\boldsymbol{P}\otimes\boldsymbol{1}_n\boldsymbol{1}_n^{\top}+\boldsymbol{I}_m\otimes\boldsymbol{T})\boldsymbol{1}_{mn}).
    \end{align*}
    First, we calculate
        \begin{align*}
            (\boldsymbol{P}\otimes\boldsymbol{1}_n\boldsymbol{1}_n^{\top}+\boldsymbol{I}_m\otimes\boldsymbol{T})\boldsymbol{1}_{mn}=&(\boldsymbol{P}\otimes\boldsymbol{1}_n\boldsymbol{1}_n^{\top}+\boldsymbol{I}_m\otimes\boldsymbol{T})\boldsymbol{1}_{m}\otimes\boldsymbol{1}_n\\
            =&\boldsymbol{P}\boldsymbol{1}_m\otimes\boldsymbol{1}_n\boldsymbol{1}_n^{\top}\boldsymbol{1}_n+\boldsymbol{I}_m\boldsymbol{1}_m\otimes\boldsymbol{T}\boldsymbol{1}_n\\
            =&n\boldsymbol{P}\boldsymbol{1}_m\otimes\boldsymbol{1}_n+\boldsymbol{1}_m\otimes\boldsymbol{T}\boldsymbol{1}_n,
        \end{align*}
    let $t_0$ be the row sum of $\boldsymbol{T}$. Then
        \begin{align*}
            \text{diag}((\boldsymbol{P}\otimes\boldsymbol{1}_n\boldsymbol{1}_n^{\top}+\boldsymbol{I}_m\otimes\boldsymbol{T})\boldsymbol{1}_{mn})=&\text{diag}(n\boldsymbol{P}\boldsymbol{1}_m\otimes\boldsymbol{1}_n)+\text{diag}(\boldsymbol{1}_m\otimes\boldsymbol{T}\boldsymbol{1}_n)\\
            =&n\text{diag}(\boldsymbol{P}\boldsymbol{1}_m\otimes\boldsymbol{1}_n)+\text{diag}(\boldsymbol{1}_m\otimes t_0\boldsymbol{1}_n)\\
            =&n\text{diag}(\boldsymbol{P}\boldsymbol{1}_m)\otimes\boldsymbol{I}_n+t_0\boldsymbol{I}_{mn},
        \end{align*}
    hence
    \begin{align*}
        \boldsymbol{B}=\boldsymbol{P}\otimes\boldsymbol{1}_n\boldsymbol{1}_n^{\top}+\boldsymbol{I}_m\otimes\boldsymbol{T}-(n\text{diag}(\boldsymbol{P}\boldsymbol{1}_m)\otimes\boldsymbol{I}_n+t_0\boldsymbol{I}_{mn}).
    \end{align*}
    Let $\boldsymbol{F}$ be the DFT matrix, as any circulant matrix can be diagonalized by a DFT matrix, denoting the eigenvalue of $\boldsymbol{T}$ is $t_0,t_1\dots,t_{n-1}$, which is real and $t_0$ is the largest as $\boldsymbol{T}$ is a symmetric circulant matrix and the elements in $\boldsymbol{T}$ are nonnegative. Then
        \begin{align*}
            \boldsymbol{I}_m\otimes\boldsymbol{F}^{-1}(\boldsymbol{P}\otimes\boldsymbol{1}_n\boldsymbol{1}_n^{\top})\boldsymbol{I}_m\otimes\boldsymbol{F}=&\boldsymbol{I}_m\boldsymbol{P} \boldsymbol{I}_m\otimes\boldsymbol{F}^{-1}\boldsymbol{1}_n\boldsymbol{1}_n^{\top}\boldsymbol{F}\\
            =&\boldsymbol{P}\otimes\begin{pmatrix} n & \\ & \boldsymbol{O}_{n-1}\end{pmatrix}\\
            =&\begin{pmatrix} n\boldsymbol{P} & \\ & \boldsymbol{O}_{mn-m}\end{pmatrix},
        \end{align*}
        \begin{align*}
             \boldsymbol{I}_m\otimes\boldsymbol{F}^{-1}(\boldsymbol{I}_m\otimes\boldsymbol{T})\boldsymbol{I}_m\otimes\boldsymbol{F}=\boldsymbol{I}_m\otimes\boldsymbol{F}^{-1}\boldsymbol{T}\boldsymbol{F}=\boldsymbol{I}_m\otimes\boldsymbol{\Lambda},
        \end{align*}
        \begin{align*}
             \boldsymbol{I}_m\otimes\boldsymbol{F}^{-1}(n\text{diag}(\boldsymbol{P}\boldsymbol{1}_m)\otimes\boldsymbol{I}_n)\boldsymbol{I}_m\otimes\boldsymbol{F}=&n\text{diag}(\boldsymbol{P}\boldsymbol{1}_m)\otimes\boldsymbol{F}^{-1}\boldsymbol{I}_n\boldsymbol{F}\\
             =&n\text{diag}(\boldsymbol{P}\boldsymbol{1}_m)\otimes\boldsymbol{I}_n,
        \end{align*}
        \begin{align*}
             \boldsymbol{I}_m\otimes\boldsymbol{F}^{-1}(t_0\boldsymbol{I}_{mn})\boldsymbol{I}_m\otimes\boldsymbol{F}=t_0\boldsymbol{I}_{mn}.
        \end{align*}
    Thus
        \begin{align*}
            \boldsymbol{I}_m\otimes\boldsymbol{F}^{-1}(\boldsymbol{B})\boldsymbol{I}_m\otimes\boldsymbol{F}=&\begin{pmatrix} n\boldsymbol{P} & \\ & 0\end{pmatrix}+\boldsymbol{I}_m\otimes\boldsymbol{\Lambda}-n\text{diag}(\boldsymbol{P}\boldsymbol{1}_m)\otimes\boldsymbol{I}_n-t_0\boldsymbol{I}_{mn}\\
            =&\begin{pmatrix} n\boldsymbol{P}-n\text{diag}(\boldsymbol{P}\boldsymbol{1}_m) & \\ & (t_i-t_0)\boldsymbol{I}_m-n\text{diag}(\boldsymbol{P}\boldsymbol{1}_m)\end{pmatrix}.
        \end{align*}
    Denoting $s_j$ the row sum of $\boldsymbol{P}$, we have
        \begin{align*}
            s_j=&\sum_{i=0}^{m-1}p^{-|i-j|}\\
            =&\sum_{i=0}^{j-1}p^{-(j-i)}+\sum_{i=j}^{m-1}p^{-(i-j)}\\
            =&\frac{1+p-p^{-j}-p^{j+1-m}}{p-1}.
        \end{align*}
    So the eigenvalue of $\boldsymbol{B}$ is
    \begin{align*}
        n\sigma(\boldsymbol{P}-\text{diag}(\boldsymbol{P}\boldsymbol{1}_m))=(p-1)p^{d-1}\sigma(\boldsymbol{P}-\text{diag}(\boldsymbol{P}\boldsymbol{1}_m))
    \end{align*}
    and
    \begin{align*}
        t_i-t_0-ns_j=t_i-t_0-p^{d-1}(1+p-p^{-j}-p^{j+1-m}),
    \end{align*}
    where $i=1,\dots,n-1;\ j=0,\dots,m-1$. Hence, the eigenvalue of $\boldsymbol{A}$ is
    \begin{align*}
        (1-p^{-1})\sigma(\boldsymbol{P}-\text{diag}(\boldsymbol{P}\boldsymbol{1}_m))
    \end{align*}
    and
    \begin{align*}
        p^{-d}(t_i-t_0)-\frac{1+p-p^{-j}-p^{j+1-m}}{p}.
    \end{align*}
    Note that $\boldsymbol{P}-\text{diag}(\boldsymbol{P}\boldsymbol{1}_m)$ is symmetric, then there exists an orthogonal matrix $\boldsymbol{Q}\in\mathbb{R}^{m\times m}$ to diagonalize it. So the corresponding eigenvector is $\boldsymbol{I}_m\otimes\boldsymbol{F}\begin{pmatrix} \boldsymbol{Q} & \\ & \boldsymbol{I}_{m(p-1)p^{d-1}-m}\end{pmatrix}$. Moreover, the spectrum of $\boldsymbol{A}$ must be compatible with the spectrum of $D$. So by Theorem \ref{thm:spectrum_D_general_m}
    we have $p^{-d}(t_i-t_0)-\frac{1+p-p^{-j}-p^{j+1-m}}{p}=-p^{i-1}-p^{i-2}+\frac{p^{-j}+p^{j+1-m}}{p}$.
    In summary, we get
        \begin{align*}
            &(\boldsymbol{I}_m\otimes\boldsymbol{F}^{-1})\begin{pmatrix} \boldsymbol{Q}^{-1} & \\ & \boldsymbol{I}_{m(p-1)p^{d-1}-m}\end{pmatrix}(\boldsymbol{A})(\boldsymbol{I}_m\otimes\boldsymbol{F})\begin{pmatrix} \boldsymbol{Q} & \\ & \boldsymbol{I}_{m(p-1)p^{d-1}-m}\end{pmatrix}\\
            =&\begin{pmatrix} (1-p^{-1})\lambda_k & \\ & -p^{i-1}-p^{i-2}+\frac{p^{-j}+p^{j+1-m}}{p}\end{pmatrix}
        \end{align*}
    where $\lambda_k$ are the eigenvalues of $\boldsymbol{P}-\text{diag}(\boldsymbol{P}\boldsymbol{1}_m)$ for $k=0,\dots,m-1$, and $i=1,\dots,d-1$, $j=0,\dots,m-1$.
\end{proof}

\begin{lemma}
    \label{lem:smallest_non-zero_eigenvalue_-A}
    The smallest non-zero eigenvalue of $-\boldsymbol{A}$ is
    \begin{align*}
        \theta_{p,m}=\begin{cases}
            (1-p^{-1}),&m=1;\\
            -(1-p^{-1})\lambda_1,&m>1.
        \end{cases}
    \end{align*}
\end{lemma}
\begin{proof}
Denoting $\boldsymbol{P}-\text{diag}(\boldsymbol{P}\boldsymbol{1}_m)=-\boldsymbol{L}$, then the eigenvalues of $\boldsymbol{L}$ is $\mu_i=-\lambda_i$ with $0=\mu_0<\mu_1<\dots<\mu_{m-1}$. As $\mu_0=0$ with eigenvector $\boldsymbol{1}_m$, so by Courant–Fischer min-max theorem, we have
\begin{align*}    \mu_1=\min_{\substack{\boldsymbol{x}\neq\boldsymbol{0}_m\\\boldsymbol{x}\bot\boldsymbol{1}_m}}\frac{\boldsymbol{x}^{\top}\boldsymbol{L}\boldsymbol{x}}{\boldsymbol{x}^{\top}\boldsymbol{x}}.
\end{align*}
Taking $\boldsymbol{x}=\boldsymbol{e}_0-\boldsymbol{e}_{m-1}$, we have
    \begin{align*}
        \frac{\boldsymbol{x}^{\top}\boldsymbol{L}\boldsymbol{x}}{\boldsymbol{x}^{\top}\boldsymbol{x}}=&\frac{1}{2}(L_{0,0}+L_{m-1,m-1}-2L_{0,m-1})\\
        =&\frac{1}{2}(2\frac{1-p^{1-m}}{p-1}+2p^{1-m})\\
        =&\frac{1+(p-2)p^{1-m}}{p-1}\\
        \leq& 1
    \end{align*}
with equality if and only if $p=2$, $\forall m$ and $m=1$, $\forall p$.\\
Therefore, the smallest nonzero eigenvalue of $-\boldsymbol{A}$ is
    \begin{align*}
        \theta_{p,m}=&\min_{\substack{1\leq k\leq m-1\\1\leq i\leq d-1\\0\leq j\leq m-1}}\left\{-(1-p^{-1})\lambda_k,\ p^{i-1}+p^{i-2}-\frac{p^{-j}+p^{j+1-m}}{p}\right\}\\
        =&\min\{-(1-p^{-1})\lambda_1,1-p^{-m}\}\\
        =&\begin{cases}
            (1-p^{-1}),&m=1;\\
            -(1-p^{-1})\lambda_1,&m>1.
        \end{cases}
    \end{align*}
\end{proof}

\begin{theorem}
When $\rho\in(0,\rho^*)$, then the solution of \eqref{eq:mean_field_equation_finite_quotient} is unique.
\end{theorem}
\begin{proof}
   As $\theta_{p,m}\leq(1-p^{-m})$ by Lemma \ref{lem:smallest_non-zero_eigenvalue_-A}, then the solution must be radially symmetric. If $\boldsymbol{u}$ and $\boldsymbol{v}$ both solve \eqref{eq:matrix_equation}. Then
    \begin{align*}
        \boldsymbol{A}(\boldsymbol{u}-\boldsymbol{v})+\rho(e^{\boldsymbol{u}}-e^{\boldsymbol{v}})=\boldsymbol{0}.
    \end{align*}
    By the mean value theorem applied for each component, there exists $\boldsymbol{\xi}$ such that
    \begin{align*}
        (\boldsymbol{A}+\rho\text{diag}e^{\boldsymbol{\xi}})(\boldsymbol{u}-\boldsymbol{v})=\boldsymbol{0}.
    \end{align*}
    Theorem \ref{thm:refined_radial_symmetric_and_monotone} gives $e^{\xi_i}\leq e^M$, which is independent of $\rho$ and $d$. Besides, as $\boldsymbol{A}$ is semi-negative definite and
    \begin{align*}
        \rho e^{\xi_i}<\rho^*e^M=\theta_{p,m},
    \end{align*}
    then $\rho e^{\xi_i}$ does not exceed the first non-zero eigenvalue of $\boldsymbol{A}$. Then $(\boldsymbol{A}+\rho\text{diag}e^{\boldsymbol{\xi}})$ is invertible, then we get the uniqueness.
\end{proof}

\begin{remark}
    The order of $\rho^*$ is optimal. Since $\sum_{i}e^{u_i}=p^{d}$, we have $m(p-1)p^{d-1}e^M\geq p^d$, therefore $M>\ln\left(\frac{p}{m(p-1)}\right)$, which means
    \begin{align*}
        \rho^*<\frac{m(p-1)}{p}\theta_{p,m}=\theta_{p,m}V.
    \end{align*}
    On the other hand, since
    \begin{align*}
        M<\begin{cases}
        \ln\left(\frac{p}{p-2}\right),\ &p>2;\\
        \ln 4,&p=2,
    \end{cases}
    \end{align*}
    which means
    \begin{align*}
        \rho^*>\begin{cases}
        \frac{1}{4}\theta_{2,m}, &p=2;\\
    (1-2p^{-1})\theta_{p,m},\ &p>2.
    \end{cases}
    \end{align*}
\end{remark}

\section{Mean field equation on the Tate curve}
Since the finite quotient converges to the Tate curve as $d\rightarrow\infty$, we investigate whether the solutions converge accordingly. From the numerical experiments, we observe that the solution on the finite quotient is convergent. Moreover, the limiting solution retains properties like uniqueness and monotonicity.

\subsection{Convergence to the Tate curve}

\begin{theorem}
     When $\rho\in(0,\rho^*)$, then there exists a solution $u$ satisfying \eqref{eq:mean_field_equation_Tate_curve}.
\end{theorem}
\begin{proof}
    For convenience, we denote $X_m=\mathbb{Q}_{p}^{\times}/p^{m\mathbb{Z}}$ with the Haar measure $d\mu=\frac{dx}{|x|_p}$ and $X_{m,d}=\mathbb{Q}_{p}^{\times}/p^{m\mathbb{Z}}/(1+p^d\mathbb{Z}_p)$ with the Haar measure $d\mu_d=p^{-d}d\delta$, which can be extended to the measure $\widetilde{\mu_d}$ on $X_m$. Clearly, $\widetilde{\mu_d}\overset{*}{\rightharpoonup}\mu$ as $d\rightarrow\infty$. Let $D_d$ denote the operator on the finite quotient. Consider
    \begin{align}
    \label{eq:mean_field_equation_on_Xmd}
        D_du_d+\rho e^{u_d}=\rho p^de_y\ \text{in}\ X_{m,d}
    \end{align}
    with
    \begin{align}
    \label{eq:int_e^ud=1}
        \int_{X_{m,d}}e^{u_d}d\mu_d=1.
    \end{align}
    Since $D_du_d=\rho p^de_y-\rho e^{u_d}$, we have $|D_du_d|\leq\rho p^de_y+\rho e^{u_d}$, then
    \begin{align}
    \label{eq:int_|Ddud|_upper_bound}
        \int_{X_{m,d}}|D_du_d|d\mu_d\leq2\rho.
    \end{align}
    Define measures $\nu_d=e^{u_d}\mu_d$ and $\omega_d=D_du_d\mu_d$ on $X_{m,d}$. As $X_{m,d}\hookrightarrow X_m$, so they can be seen as the restriction of measure $\widetilde{\nu_d}$ and $\widetilde{\omega_d}$ on $X_m$, i.e. $\widetilde{\nu_d}|_{X_{m,d}}=\nu_d$ and $\widetilde{\omega_d}|_{X_{m,d}}=\omega_d$. On the other hand, the function on $X_{m,d}$ can extend to a locally constant function on $X_m$.

    Since $X_m$ is compact, we have $C(X_m)^*=\mathscr{M}(X_m)$, so by the Banach-Alaoglu theorem, $\mathscr{M}(X_m)$ is weak-* compact. Then from \eqref{eq:int_e^ud=1} and \eqref{eq:int_|Ddud|_upper_bound}, we have $||\widetilde{\nu_d}||_{\mathscr{M}}=1$ and $||\widetilde{\omega_d}||_{\mathscr{M}}\leq2\rho$, so there exists $\nu$, $\omega\in \mathscr{M}(X_m)$ and a subsequence $d_k$ such that
    \begin{align*}
        \widetilde{\nu_{d_k}}\overset{*}{\rightharpoonup}\nu,\ \widetilde{\omega_{d_k}}\overset{*}{\rightharpoonup}\omega
    \end{align*}
    as $\nu_d<<\mu_d$ and $\frac{d\nu_d}{d\mu_d}=e^{u_d}\leq e^M$ uniformly in $d$ by Theorem \ref{thm:refined_radial_symmetric_and_monotone}, we have $\nu<<\mu$, denoting $\nu=e^{\hat{u}}\mu$ for some function $\hat{u}$.

    Taking any locally constant function $\phi$ on $X_m$, when $d$ is large enough, $\phi$ is well-defined on $X_{m,d}$ as $X_m$ is compact. Multiply \eqref{eq:mean_field_equation_on_Xmd} by $\phi$ and integrate over $X_{m,d}$, we have
    \begin{align*}
        \int_{X_{m,d}}(D_du_d)\phi d\mu_d+\rho\int_{X_{m,d}}e^{u_d}\phi d\mu_d=\rho\phi(y),
    \end{align*}
    and it can be rewritten as
    \begin{align*}
        \int_{X_{m}}\phi d\widetilde{\omega_{d}}+\rho\int_{X_{m}}\phi d\widetilde{\mu_{d}}=\rho\phi(y).
    \end{align*}
    By taking a limit of the subsequence $d_k$, we get
    \begin{align}
    \label{eq:measure_equation_on_Xm}
        \int_{X_{m}}\phi d\omega+\rho\int_{X_{m}}\phi d\nu=\rho\phi(y).
    \end{align}
    
    Decompose $u_d=v_d+\rho G_d$, where $G_d$ is the Green's function on finite quotient, then $v_d$ satisfies
    \begin{align*}
        D_dv_d=-\rho e^{u_d}+\frac{\rho}{V}=-\rho e^{\rho G_d}e^{v_d}+\frac{\rho}{V},
    \end{align*}
    so
    \begin{align*}
        v_d(x)-\overline{v_d}=\int_{X_{m,d}}G_d(x,y)\left(-\rho e^{u_d(y)}+\frac{\rho}{V}\right)d\mu_d(y),
    \end{align*}
    where
    \begin{align*}
        \overline{v_d}=\frac{1}{|X_{m,d}|}\int_{X_{m,d}}v_d(y)d\mu_d(y)=\frac{1}{m(1-p^{-1})}\int_{X_{m,d}}v_d(y)d\mu_d(y).
    \end{align*}

    Since
    \begin{align}
    \label{eq:Green's_function_series_form}
    G(x,y)=\lambda_0\text{log}_pd(x,y)+\sum_{n=0}^{\infty}\lambda_nd(x,y)^n+C(|x|_p,|y|_p)
    \end{align}
    and $G_d(x,y)\rightarrow G(x,y)$ pointwise in $X_m\backslash\{y\}$ \cite{huang2025greensfunctiontatecurve}, then for $\epsilon=|x|_pp^{-N}$, we decompose
    \begin{align*}
        \int_{X_{m}}|G(x,y)|d\mu(y)=\int_{B_{\epsilon}(x)}|G(x,y)|d\mu(y)+\int_{X_{m}\backslash B_{\epsilon}(x)}|G(x,y)|d\mu(y).
    \end{align*}
     And we have
    \begin{align*}
         \int_{B_{\epsilon}(x)}\text{log}_pd(x,y)d\mu(y)=&\int_{x+p^{v_p(x)+N}\mathbb{Z}_p}\text{log}_pd(x,y)d\mu(y)\\
         =&\int_{\mathbb{Z}_p}\text{log}_p(p^{-N}|t|_p)p^{-N}dt\\
         =&-\left(N+\frac{1}{p-1}\right)p^{-N}
    \end{align*}
    and
    \begin{align*}
        \int_{B_{\epsilon}(x)}d(x,y)^nd\mu(y)=&\int_{x+p^{v_p(x)+N}\mathbb{Z}_p}d(x,y)^nd\mu(y)\\
         =&\int_{\mathbb{Z}_p}(p^{-N}|t|_p)^np^{-N}dt\\
         =&\frac{p^n(p-1)}{p^{n+1}-1}p^{-N(n+1)}.
    \end{align*}
    So we have
    \begin{align*}
        \int_{B_{\epsilon}(x)}|G(x,y)|d\mu(y)\leq CNp^{-N}\leq C'\epsilon|\ln\epsilon|,
    \end{align*}
    thus $G\in L^1(X_m)$. Additionally, as $-\rho e^{u_d(y)}+\frac{\rho}{V}$ is bounded and $|\widetilde{G_d}(\cdot,y)|\leq |G(\cdot,y)|$ \cite{huang2025greensfunctiontatecurve}, we have
    \begin{align*}
        |v_d(x)-\overline{v_d}|\leq C||G_d(\cdot,y)||_{L^1(X_{m,d})}=C||\widetilde{G_d}(\cdot,y)||_{L^1(X_{m})}\leq C||G(\cdot,y)||_{L^1(X_{m})}.
    \end{align*}
    Then from \eqref{eq:int_e^ud=1}, we have
    \begin{align*}
        1=\int_{X_{m,d}}e^{u_d}d\mu_d=e^{\overline{v_d}}\int_{X_{m,d}}e^{\rho G_d}e^{v_d-\overline{v_d}}d\mu_d,
    \end{align*}
    since
    \begin{align*}
        e^{-C}\leq e^{v_d-\overline{v_d}}\leq e^{C}
    \end{align*}
    and
    \begin{align*}
        C_1\leq\int_{X_{m,d}}e^{\rho G_d}d\mu_d\leq C_2
    \end{align*}
    by \eqref{eq:Green's_function_series_form}, so $\overline{v_d}\leq C$ uniformly in $d$. Therefore, $v_d\in L^{\infty}(X_{m,d})$ as $|v_d|\leq|v_d-\overline{v_d}|+|\overline{v_d}|$, thus $\widetilde{v_d}\in L^{\infty}(X_{m})$.

    Using the self-adjointness of $D$, we have
        \begin{align*}
            \int_{X_{m,d}}(D_du_d)\phi d\mu_d=&\int_{X_{m}}\widetilde{D_du_d}\phi d\widetilde{\mu_d}\\
            =&\int_{X_{m}}(D_d\widetilde{u_d})\phi d\widetilde{\mu_d}\\
            =&\int_{X_{m}}(D_d\phi)\widetilde{u_d} d\widetilde{\mu_d}\\
            =&\int_{B_{\epsilon}(y)}(D_d\phi)\widetilde{u_d} d\widetilde{\mu_d}+\int_{X_{m}\backslash B_{\epsilon}(y)}(D_d\phi)\widetilde{u_d} d\widetilde{\mu_d}\\
            =&\mathrm{I}+\mathrm{II},
        \end{align*}
    where
        \begin{align*}
            |\mathrm{I}|=&\left|\int_{B_{\epsilon}(y)}(D_d\phi)\widetilde{u_d} d\widetilde{\mu_d}\right|\\
            \leq&\rho\int_{B_{\epsilon}(y)}|D_d\phi(x)||\widetilde{G_d}(x,y)|d\widetilde{\mu_d}(x)+\int_{B_{\epsilon}(y)}|D_d\phi||\widetilde{v_d}|d\widetilde{\mu_d}\\
            \leq& C_1\epsilon|\ln\epsilon|+C_2\epsilon
        \end{align*}
    and we have
    \begin{align*}
        \mathrm{II}=\int_{X_{m}\backslash B_{\epsilon}(y)}(D_d\phi)\widetilde{u_d} d\widetilde{\mu_d}=\int_{X_{m}\backslash B_{\epsilon}(y)}(D\phi)\widetilde{u_d} d\widetilde{\mu_d}
    \end{align*}
    for sufficiently large $d$.

    And in $X_{m}\backslash B_{\epsilon}(y)$, since $\widetilde{u_d}(x)$ depend only on $|x|_p$, $|y|_p$ and $|x-y|_p$, so $\forall\ x_0\in X_{m}\backslash B_{\epsilon}(y)$, $\exists\ r_0(x_0)<\min\{|x_0|_p,|x_0-y|_p\}$, s.t. $\forall\ x_0\in B(x_0,r_0(x_0))$, we have $|x|_p=|x_0|_p$, $|x-y|_p=|x_0-y|_p$. So $\widetilde{u_d}(x)$ is locally constant on $B(x_0,r_0(x_0))$. Since $X_{m}\backslash B_{\epsilon}(y)$ is compact, so $X_{m}\backslash B_{\epsilon}(y)=\bigcup_{i=1}^{K}B(x_i,r_i(x_i))$. Taking $\delta=\min\{r_1(x_1),\dots,r_K(x_K)\}$, then $\forall\ x,z\in X_{m}\backslash B_{\epsilon}(y)$ with $|x-z|_p<\delta$, we have $\widetilde{u_d}(x)=\widetilde{u_d}(z)$ as balls on $p$-adic field are either disjoint or concentric. And we notice that $|\widetilde{u_d}|\leq C_{\epsilon}$, so there exists a subsequence $\widetilde{u_{d_k}}$ converges to a locally constant function $u$ on $X_{m}\backslash B_{\epsilon}(y)$ such that
    \begin{align*}
        \int_{X_{m}\backslash B_{\epsilon}(y)}(D\phi)\widetilde{u_{d_k}} d\widetilde{\mu_{d_k}}=\int_{X_{m}\backslash B_{\epsilon}(y)}(D\phi)(\widetilde{u_{d_k}}-u) d\widetilde{\mu_{d_k}}+\int_{X_{m}\backslash B_{\epsilon}(y)}(D\phi)u d\widetilde{\mu_{d_k}}.
    \end{align*}
    The first term can be controlled by $||\widetilde{u_{d_k}}-u||_{L^{\infty}(X_{m}\backslash B_{\epsilon}(y))}$, so
    \begin{align*}
        \mathrm{II}\rightarrow \int_{X_{m}\backslash B_{\epsilon}(y)}(D\phi)u d\mu.
    \end{align*}
    Letting $\epsilon\rightarrow0$ and by the standard diagonal argument, we obtain a convergent function $u$, so
    \begin{align*}
        \int_{X_{m}}\widetilde{D_{d_k}u_{d_k}}\phi d\widetilde{\mu_{d_k}}\rightarrow\int_{X_{m}}(D\phi)u d\mu=\int_{X_{m}}(Du)\phi d\mu
    \end{align*}
    as $k\rightarrow\infty$. Hence $\omega=Du\mu$ in the sense of distributions.

    What's more, since
    \begin{align*}
        \int_{X_m}\phi d\nu=\lim_{k \to\infty}\int_{X_{m}}\phi e^{\widetilde{u_{d_k}}}d\widetilde{\mu_{d_k}}=\int_{X_m}\phi e^{u}d\mu,
    \end{align*}
    and by definition, we have
    \begin{align*}
        \int_{X_m}\phi d\nu=\int_{X_m}\phi e^{\hat{u}}d\mu,
    \end{align*}
    so $\hat{u}=u$ for a.e. $x$. Then by \eqref{eq:measure_equation_on_Xm} we get
    \begin{align*}
        \int_{X_{m}}(Du)\phi d\mu+\rho\int_{X_m}e^{u}\phi d\mu=\rho\phi(y).
    \end{align*}
    Lastly, as $u(x)$ is locally constant on $X_m\backslash\{y\}$, thus it satisfies \eqref{eq:mean_field_equation_Tate_curve}.
\end{proof}

Since the solution we found converges from the solution on finite quotient, it inherits the properties of the solution on finite quotient. So the solution $u$ has those radially symmetric and monotone properties with a universal upper bound $M$.

When $\rho\geq\rho^*$, then there may exist non-monotone radially symmetric solutions, let alone a non-radially symmetric one, so that $u$ will not have the universal upper bound, this may prevent the convergence of solutions on the finite quotient, and hence the existence cannot be guaranteed.

\subsection{Uniqueness of the solution}
Although the solution on finite quotient is unique when $\rho\in(0,\rho^*)$, we can't ensure uniqueness on the Tate curve as there may has another solution on the Tate curve but vanishes on finite quotient. We now address uniqueness on the Tate curve by a standard method \cite{MR1283323}. First, we show that the solution is unique when $\rho\leq\varepsilon$. Next, we prove the linearized solution is non-degenerate under the same constraints of $\rho$, then we can get the uniqueness.

\begin{lemma}
    \label{lem:uniqueness_small_rho}
    There exists $\varepsilon>0$, such that when $\rho\leq\varepsilon$, then \eqref{eq:mean_field_equation_Tate_curve} has a unique solution.
\end{lemma}
\begin{proof}
    Let $u_n$, $v_n$ be distinct solutions of \eqref{eq:mean_field_equation_Tate_curve} with $\rho_n\rightarrow0$. Since $\int_{\mathbb{Q}_{p}^{\times}/p^{m\mathbb{Z}}} e^u\frac{dx}{|x|_p}=1$, the difference $u_n-v_n$ must change sign. Define
    \begin{align*}
        \phi_n=\frac{u_n-v_n}{\|u_n-v_n\|_{L^{\infty}}}.
    \end{align*}
    Then $\phi_n$ satisfies
    \begin{align*}
        D\phi_n+V_n(x)\phi_n=0,
    \end{align*}
    where $V_n(x)=\rho_n\frac{e^{u_n}-e^{v_n}}{u_n-v_n}$. Since $u(x)\leq M$ by Theorem \ref{thm:refined_radial_symmetric_and_monotone}, $V_n(x)$ converges to $0$. And it is easy to see that $\phi_n$ uniformly converges to the eigenfunction of $D$ with the eigenvalue $0$, which is the harmonic function on $\mathbb{Q}_{p}^{\times}/p^{m\mathbb{Z}}$. As $\mathbb{Q}_{p}^{\times}/p^{m\mathbb{Z}}$ is compact without boundary, then the harmonic function on it must be a constant function, which contradicts the fact that $\phi_n$ must change the sign.
\end{proof}

Next, we consider the linearized equation of \eqref{eq:mean_field_equation_Tate_curve} as
\begin{align}
\label{eq:linearized_equation}
    Df+\rho e^uf=0.
\end{align}
If \eqref{eq:linearized_equation} is nondegenerate for $\rho\leq\rho^*$, then Lemma \ref{lem:uniqueness_small_rho} implies the uniqueness of the solution for such $\rho$. To address it, we give the following Lemma:

\begin{lemma}
\label{lem:non-degenerate_operator}
    Let $L$ be an operator with eigenvalues $\lambda_1\leq \lambda_2\leq\cdots$, and $\exists\ k$ s.t. $\lambda_k<V(x)<\lambda_{k+1}$. Then the operator $L-V(x)$ is nondegenerate.
\end{lemma}
\begin{proof}
    Suppose that the eigenvalue of $L-V(x)$ is $\mu_1\leq \mu_2\leq\cdots$. Then we have operator inequalities
    \begin{align*}
        L-\lambda_{k+1}<L-V(x)<L-\lambda_{k}.
    \end{align*}
    Then we have
    \begin{align*}
        \lambda_{i}-\lambda_{k+1}<\mu_i<\lambda_{i}-\lambda_{k}
    \end{align*}
    for all $i$. Taking $i=k$ and $i=k+1$, we get
    \begin{align*}
        \mu_k<0<\mu_{k+1},
    \end{align*}
    which means $0$ is not the eigenvalue of $L-V(x)$. Thus the operator $L-V(x)$ is nondegenerate.
\end{proof}

\begin{theorem}
When $\rho\in(0,\rho^*)$, then the solution of \eqref{eq:mean_field_equation_Tate_curve} is unique.
\end{theorem}
\begin{proof}
    By Theorem \ref{thm:refined_radial_symmetric_and_monotone}, we can see that $0<\rho e^u<\rho^*e^M=\theta_{p,m}$. And from Theorem \ref{thm:spectrum_D_general_m} and Lemma \ref{lem:smallest_non-zero_eigenvalue_-A}, the smallest nonzero eigenvalue of $-D$ is
\begin{equation}
    \begin{aligned}
        &\min_{\substack{1\leq k\leq m-1\\i\geq 1\\0\leq j\leq m-1}}\left\{-(1-p^{-1})\lambda_k,\ p^{i-1}+p^{i-2}-\frac{p^{-j}+p^{j+1-m}}{p}\right\}\\
        =&\min\{-(1-p^{-1})\lambda_1,1-p^{-m}\}\\
        =&\begin{cases}
            (1-p^{-1}),&m=1\\
            -(1-p^{-1})\lambda_1,&m>1
        \end{cases}\\
        =&\theta_{p,m}.
    \end{aligned}
\end{equation}
Then from Lemma \ref{lem:non-degenerate_operator}, the linearized operator is nondegenerate, then we obtain the uniqueness.
\end{proof}

\section{Mean field equation on Elliptic curves with good reduction}

Given an elliptic curve $E$ over a non-Archimedean field, it either has good reduction or is a Tate curve at any non-Archimedean places after taking finite field extension. First consider $E(\mathbb{Q}_p)$, which can be generalized to the finite extension of $\mathbb{Q}_p$. And consider the reduction group homomorphism
\begin{align*}
    \varphi:E(\mathbb{Q}_p)&\to E(\mathbb{F}_p)\\
    x&\mapsto\bar{x}
\end{align*}
then $\ker\varphi\cong p\mathbb{Z}_p$, where $\cong$ denotes analytic isomorphism.
Huang defined such operator $\mathcal{D}$ on $E(\mathbb{Q}_p)$ as
\begin{align*}
    \mathcal{D}u(x)=c_p\int_{\substack{z\in E(\mathbb{Q}_p)\\ \bar{z}=\bar{x}}}\left(\frac{1}{|z-x|_p^2}+c\right)(u(x)-u(z))dz+d\int_{\substack{z\in E(\mathbb{Q}_p)\\ \bar{z}\neq\bar{x}}}(u(x)-u(z))dz,
\end{align*}
where $c_p=\frac{p(p-1)}{p+1}$, $c=\frac{p+1-N}{N}$, $d=\frac{p(p-1)}{N}$ and $N=\#E(\mathbb{F}_p)$. Then we have $c_p(1+c)=d$ and $V=|E(\mathbb{Q}_p)|=\#E(\mathbb{F}_p)|p\mathbb{Z}_p|=\frac{N}{p}=\frac{p-1}{d}$. One can check that the operator $\mathcal{D}$ is self-adjoint and positive semi-definite under the inner product $(f,g)=\int_{E(\mathbb{Q}_p)}f(x)g(x)dx$. 

They showed that the Green's function $h(x,y)$ of $\mathcal{D}$ is the Néron local height function on $E(\mathbb{Q}_p)$ with respect to $y$, i.e.
\begin{align*}
    \mathcal{D}h(x,y)=\delta_y(x)-\frac{1}{V}.
\end{align*}

Therefore, we can promote our results to the Elliptic curve with good reduction. Consider the mean field equation
\begin{align*}
    -\mathcal{D}u+\rho e^u=\rho\delta_y \text{ in }E(\mathbb{Q}_p).
\end{align*}
Consider the quotient space $E(\mathbb{Q}_p)/p^d\mathbb{Z}_p$. Since $E(\mathbb{Q}_p)$ is compact, then $\#E(\mathbb{Q}_p)/p^d\mathbb{Z}_p<\infty$, so similarly, the operator $\mathcal{D}$ on this quotient space reduces to a graph Laplacian on a finite weighted graph. Then from the results in \cite{huang2020existence} we get the existence of the solution. And similar from the arguments in Chapter 3 and Chapter 4, we can get the existence and uniqueness result by investigating the structure of the solution.
\begin{remark}
    In the study of Tate curves (corresponding to bad reduction), since the curve is parametrized by the $p$-adic multiplicative group $\mathbb{Q}_p^{\times}$, the chosen finite quotient space is obtained by taking the quotient by the higher unit group (multiplicative subgroup) $1+p^d\mathbb{Z}_p$; whereas for an elliptic curve $E(\mathbb{Q}_p)$ with good reduction, the formal group of its reduction kernel is locally analytically isomorphic to the additive group, and hence we quotient out the additive subgroup $p^d\mathbb{Z}_p$.
    
    Although superficially one quotients out an additive group while the other quotients out a multiplicative group, in the theory of $p$-adic Lie groups these two are locally completely identical as analytic objects. The classical $p$-adic exponential map furnishes a strict analytic isomorphism from the additive group to the higher unit group:
    \begin{align*}
        \exp:p^d\mathbb{Z}_p&\to1+p^d\mathbb{Z}_p\\
        x&\mapsto e^{px}
    \end{align*}
     Consequently, whether we quotient by $p^d\mathbb{Z}_p$ or by $1+p^d\mathbb{Z}_p$, the geometric essence is entirely the same: namely, one employs the chain of analytic subgroups of the local neighborhood (a $p$-adic disk) of the manifold to perform a truncation, thereby discretizing the continuous geometric space together with the pseudo-differential operator into matrix equations that can be rigorously solved on finite weighted graphs. As $d\to\infty$, both schemes converge back to their respective underlying continuous compact manifolds.
\end{remark}

\section*{Acknowledgement}
I am deeply indebted to An Huang for his patient guidance, insightful discussions, and unwavering support throughout this research. I extend my sincere gratitude to my supervisor Bobo Hua for his valuable suggestions and help.

\bibliographystyle{plain}
\bibliography{references}

\noindent Yaojia Sun, 24210180117@m.fudan.edu.cn\\
\emph{School of Mathematical Sciences, Fudan University, Shanghai, 200433, P.R. China}\\[-8pt]
\restoregeometry
\end{document}